\documentclass[12pt]{iopart}
\usepackage{amssymb, amsthm}
\usepackage{amsmath}
\usepackage{graphicx}
\allowdisplaybreaks

\newtheorem{theorem}{Theorem}[section]
\newtheorem{corollary}[theorem]{Corollary}
\newtheorem{lemma}{Lemma}[section]
\newtheorem{proposition}{Proposition}[section]
\theoremstyle{definition}
\newtheorem{definition}{Definition}[section]
\theoremstyle{definition}
\newtheorem{example}{Example}[section]

\begin{document}

\title{Where and When Orbits of Strongly Chaotic Systems Prefer to Go}

\author{M. Bolding and L. A. Bunimovich}

\address{School of Mathematics, Georgia Institute of Technology, Atlanta, GA, 30332-0160, USA}
\vspace{10pt}
\begin{indented}
\item[]E-mail: bunimovich@math.gatech.edu and mbolding3@gatech.edu
\end{indented}
\vspace{10pt}
\begin{indented}
\item[]\today
\end{indented}

\begin{abstract}
We prove that transport in the phase space of the "most strongly chaotic" dynamical systems has three different stages. Consider a finite Markov partition (coarse graining) $\xi$ of the phase space of such a system. In the first short times interval there is a hierarchy with respect to the values of the first passage probabilities for the elements of $\xi$ and therefore finite time predictions can be made about which element of the Markov partition trajectories will be most likely to hit first at a given moment. In the third long times interval, which goes to infinity, there is an opposite hierarchy of the first passage probabilities for the elements of $\xi$ and therefore again finite time predictions can be made. In the second intermediate times interval there is no hierarchy in the set of all elements of the Markov partition. We also obtain estimates on the length of the short times interval and show that its length is growing with refinement of the Markov partition which shows that practically only this interval should be taken into account in many cases. These results demonstrate that finite time predictions for the evolution of strongly chaotic dynamical systems are possible. In particular, one can predict that an orbit is more likely to first enter one subset of phase space than another at a given moment in time. Moreover, these results suggest an algorithm which accelerates the process of escape through "holes" in the phase space of dynamical systems with strongly chaotic behavior.\\

\end{abstract}

\section{Introduction}

This paper belongs to a new direction in dynamical systems theory, which originated in the theory of open dynamical systems \cite{PY,DY}. A standard set up in open systems theory includes a "hole" $H$ which is a positive measure subset of the phase space $M$ of some dynamical system generated by a map $T$ preserving a Borel probability measure $\mu$.

When a trajectory hits the hole $H$ it escapes and is not considered any more. In this setting, it is natural to assume that the map $T$ is ergodic. Otherwise, instead of a given hole one should consider its intersections with ergodic components and take each ergodic component as the phase space of an open system. In what follows we always assume that $T$ is ergodic, and therefore almost all trajectories will eventually escape. Let $P_s(n)$ be the probability that a trajectory does not escape until time $n$ (computed with respect to $\mu$), which is called the survival probability.  It is natural to ask what the decay rate of $P_s(n)$ is. This decay rate $\rho(H) = \lim_{n\rightarrow\infty}-\frac{1}{n}\log P_s(n)$ is called the escape rate.

Traditionally, the theory of open dynamical systems dealt with small holes \cite{PY,DY}. Therefore such open systems could be (and were) treated as small perturbations of the corresponding closed systems \cite{DY,KL}. In fact, the first paper about open systems \cite{PY} limited its scope to the dynamics (of billiards) with a \textit{small} hole (in the billiard table). Hence the interest was always related to the limit when the size of the hole tends to zero \cite{PY,DY}, besides some special examples when everything is easily computable (some examples can be found in the beautiful review \cite{DY}).

At the same time, a seemingly natural question on how the process of escape depends on the position of the hole in phase space had been overlooked. The second author raised this question inspired by remarkable experiments with atomic billiards \cite{FKCD,MHCR}. Moreover, this question addresses finite, rather than infinitesimal, holes. Indeed, in real systems "holes" are finite, and it is a challenge to the modern theory of dynamical systems to handle this situation.

Originally this question was formulated as follows: "How does escape rate depend on the position of the hole?" \cite{BY}. So it referred to the escape rate  $\rho(H)$. Observe that the definition of the escape rate involves a limit as time goes to infinity. Clearly this question makes sense if the measure $\mu$ is invariant, ergodic, and absolutely continuous. Indeed, if e.g. $\mu$ is sitting on some subset (hole) then the escape rate through this hole is infinite, while it could assume various finite values for subsets not belonging to this hole. (In fact, it was shown that in general the escape rate can even behave locally as a devil staircase \cite{DW}). 

A standard and natural approach to attacking a new type of a problem or question is to consider a class of systems for which an answer seems possible. Such class of dynamical systems with the strongest chaotic properties was studied in \cite{BY,B}. These dynamical systems are ergodic with respect to an absolutely continuous measure $\mu$ and have such finite Markov partition that the corresponding symbolic representation is a full Bernoulli shift. Moreover, each element in this partition has the same measure, and therefore all entries of the transition probabilities matrix are also equal each other.(In what follows such Markov partition will be called a basic Markov partition). 
Hence the measure of any element of the basic Markov partition and all transition probabilities are equal to $1/q$, where $q$ is a number of elements in the Markov partition $\xi$. Therefore the evolution of such dynamical systems is equivalent to the evolution of independent identically distributed (IID) random variables. All values of such random variables have the same probability. A typical example of random trials generating such IIDs is the throwing of a fair dice (with $q$ faces). Therefore this class of systems was called in \cite{B} fair-dice-like (FDL) systems. 
The FDL-systems form a narrow subclass with the most uniform hyperbolicity in the class of chaotic (hyperbolic) dynamical systems \cite{BP}.
One of the simplest examples of a FDL-system is the doubling map $x \rightarrow 2x$(mod1) of the unit interval. The Lebesgue measure (length) is invariant. Consider the Markov partition into intervals $(0,1/2)$ and $(1/2,1)$. Then $q=2$ and the measure of each interval and all four transition probabilities equal 1/2.

It is a standard approach in the theory of chaotic dynamical systems to pose questions about what happens in the limit when time goes to infinity, or after averaging some observables over an infinite time interval. The evolution of strongly chaotic dynamical systems is in many respects similar to the evolution of random (stochastic) processes. Therefore in the metric (ergodic) theory of dynamical systems, the main problems are about mixing (i.e. vanishing of correlations in the limit when time goes to infinity), a rate of mixing (correlation decay), and the central limit theorem (CLT) and other limit theorems which again involve a limit when time goes to infinity.

Likewise, all major characteristics of chaotic dynamical systems involve either taking a limit when time goes to infinity or averaging over an infinite time interval. Indeed, look at the definitions of Lyapunov exponents and various entropies, among others. Observe that the definition of the escape rate also involves a limit when time goes to infinity. Analogously in Nonequilibriun Statistical Mechanics, definitions of transport coefficients involve taking a limit when time goes to infinity and averaging over an infinite time interval.  Therefore the main result of \cite{BY} came as a complete surprise. Namely, it was proven that for some subclass of FDL systems not only are the escape rates generally different for different subsets (considered to be holes), but also the relations between corresponding survival probabilities can be established for \textit{all moments of time}. Namely, either all the survival probabilities for two holes of the same measure are equal or they coincide only on a short time interval after which all survival probabilities for one hole exceed the ones for another. Therefore there exists a finite moment in time when the process of escape through one hole becomes more intense than for another hole (and this moment is exactly and easily computable for FDL systems).

This result is of an essentially different nature than the ones we are used to in dynamics. Indeed, it deals with finite times rather than with the infinite time limit. Why are some specific finite moments in time important for dynamics (more precisely in this case, for transport in the phase space)? These FDL systems are the most uniformly hyperbolic (chaotic). Why is their dynamics not uniform?  This result was generalized to the entire class of FDL systems in \cite{B}. The paper \cite{BB} contains generalizations for Markov chains. Topological analogs of these results were proved in \cite{AB} where the focus was on applications to computer simulations of real systems and networks. In \cite{AB} long time \textit{estimates} were obtained for survival and first hitting probabilities with respect to (not necessarily invariant) Lebesgue measure, which is typically used as an initial distribution in numerical experiments.

The main result of \cite{BY} gives hope that it might be possible to develop a rigorous theory regarding the finite time evolution of strongly chaotic dynamical systems. The first step was to realize that one can use the ideas developed in open systems theory to study transport in the phase space of closed dynamical systems. Indeed, one can make different holes in the phase space of a closed dynamical systems and "look" through these holes at the dynamics of a closed system (like one looks inside through windows). This really represents a turnabout, because here we use open systems built from a closed system to study the dynamics of the original closed system.  This is totally different from a standard approach in the theory of open systems which does the opposite. 

The idea of "spying" on closed systems by making holes (windows) proved to be efficient and allowed to obtain various new formulas and results useful for applications (see e.g. \cite{BD,D}). It is notable (although natural) that the main impact in a number of follow up papers was made so far not by the main result of \cite{BY} but by one of its byproducts dealing with infinitesimally small holes. Consider a sequence of shrinking holes converging to some point in the phase space. Again one can place this sequence in neighborhoods of different points pursuing an answer to the same question regarding how such placement impacts the escape rate. This was another non-standard question raised in \cite{BY}. Indeed, the escape rate through a point obviously equals zero because a measure in the open systems theory is always assumed to be absolutely continuous. It was shown in \cite{BY} that by normalizing the escape rates of holes which shrink to a point by measures of these holes, one gets a limiting value. Moreover this value varies over different points. Therefore even local escape depends upon the position (and other dynamical characteristics) of the point in phase space. This result about small holes (as well as the results for large holes) was presented at the workshop in the Boltzmann institute in the summer 2008. Immediately \cite{KL,FP} it was generalized to much larger classes by leading experts in open dynamical systems. Now it is an active area because relevant perturbation techniques were already well developed.

But what to do about this new, unexpected, and strange result on large holes? A natural approach would be to generalize the main result of \cite{BY,B} on survival probabilities to a larger class of chaotic hyperbolic systems. It is always the case that when something new is found for a narrow class of systems, the results are generalized for larger and larger classes.  However, the main result of \cite{BY} gave hope that something more ambitious would be possible, namely finite time predictions for strongly chaotic dynamical systems. Observe that this main result \cite{BY}, although giving some exact values in time when the survival probabilities for different subsets of the phase space split, does not allow for finite time predictions for evolution of a system. Indeed, by comparing two subsets of the phase space of a FDL system we can only say that it is more probable (over the \textit{infinite} time interval after a certain moment) that trajectories would enter one subset compared to another. So this result does not allow us to make finite time predictions. Therefore leading experts in open systems (as well as other mathematicians) did not move into this new area of research because it was not clear what to do next.

It is a main goal of our paper to present the first rigorous results in the mathematical theory of the finite time dynamics (FTD) of (strongly) chaotic systems. It is the next needed step in this new area. In fact these new rigorous results allow one to make finite time predictions for transport in the phase space of fair-dice-like dynamical systems.  Actually, predictions can even be made for the next moment of time, i.e. for an "immediate future". (Some results of the present paper were announced in \cite{BV}. Here proofs are given of those claims as well as of several other statements). 

Recall that the $m$th order refinement of a partition $\xi$ is the partition generated by intersection of all preimages of $\xi$ from orders $0$ to $m-1$, i.e the partition generated by the sets $C_\xi^{(j_0)} \cap T^{-1} C_\xi^{(j_1)} \cap \dots \cap T^{-m+1}C_\xi^{(j_{m-1})}$ where $C_\xi^{(k)}$ is some element of the partition $\xi$. It is a well known fact that any refinement of a Markov partition is also a Markov partition.

Our main result says the following. Let $A$ and $B$ be elements of some refinement of a basic Markov partition of a FDL system. Then either the infinite sequences of their first hitting (first passage) probabilities $P_h(A,n)$ and $P_h(B,n)$ coincide, or the entire infinite interval of (positive) times gets partitioned into three subintervals. In the first very short interval the first hitting probabilities coincide. Then in the second interval (of finite length) the first hitting probabilities for $A$ exceed those for $B$.  In the third (infinite) interval the opposite inequalities hold. 
If $A$ and $B$ are elements of different refinements of the Markov partition of a FDL system then the first (short) time interval disappears because $A$ and $B$ have different measures and only the second and the third intervals remain where there are hierarchies.  

To understand the following results better, imagine that for each element we construct a piece-wise linear curve connecting the values of the corresponding first hitting probabilities $P_h(A,n)$ and $P_h(A, n+1)$ for all $n>0$.  It follows from ergodicity that  $\sum _{n=1}^\infty P_h(H,n)$=1 for any set $H$ of positive measure. Therefore if the first hitting probability curves for two subsets do not coincide then they must intersect. The main result establishes that (besides possibly a very short initial interval of coincidence) \textit{there is only one point of intersection of these curves}.

These results are much stronger than those found in \cite{BY,B}, which can be easily deduced from the results of the present paper. First of all, the following formula holds
\begin{align*}
P_s(A,n) = \sum_{m=n+1}^\infty P_h(A,m).
\end{align*}
Therefore clearly the results on comparison of elements (first hitting probabilities) of infinite series obtained in this paper are stronger than the results on comparison of the sums of such series (survival probabilities) \cite{BY,B}.

Consider now all elements of some refinement of the basic Markov partition. (All these elements have of course the same measures.) Then the main result of the present paper implies that the evolution of any FDL system consists of three stages. At the first stage, which we refer to as the short times interval, there is a hierarchy of the first hitting probabilities for different elements. At the second stage all these curves intersect. After the last such pairwise intersection, the third stage emerges which occupies what we call the long times interval, having infinite length. The intermediate interval between the first and the last intersections of the first hitting probability curves we refer to as the intermediate times interval. Observe that a standard approach in dynamical systems theory would consider only the infinite time limit whereas this partition into three time intervals is something completely new.

A crucial question about the possible practical applications of our results is what happens to the lengths of the finite short times interval and of the intermediate interval when we consider higher order refinements of the basic Markov partition. Practically speaking, it means that we analyze transport in the phase space at finer and finer scales. We prove that the length of the short times interval increases at least linearly with the order of refinement of the Markov partition. In fact numerical experiments, which we also present below, show that this growth is actually exponential. However, a principal fact is that the length of the short times interval tends to infinity when we consider transport in the phase space at finer and finer scales. Indeed, observe that the hierarchy of the first hitting probabilities in the short times interval is \textit{opposite} to the hierarchy in the (infinite) long times interval. Therefore the traditional approach to the studies of transport in the phase space of chaotic systems, which is based on time-asymptotic analysis, seems to be not quite appropriate for practical use. Indeed any analysis (via experiments and observations) of real systems lasts only a finite time. Therefore it essentially belongs to our short times interval where the dynamics/transport has quite different characteristics than in the infinite long times interval. Hence the strategy for analyzing experimental data should perhaps be reconsidered. 

It is worthwhile to mention that numerical experiments with dispersing billiards confirmed the existence of different stages in the transport of chaotic systems \cite{BV}. It should be noted however that these numerical computations can not confirm that the corresponding curves of the first hitting probabilities have only one (or even a finite number) of intersections. In fact we believe that there are more intersections than just one for billiards studied numerically in \cite{BV}. Nevertheless, these numerical simulations suggest that there are rather long intervals with alternating hierarchies of the first hitting probabilities curves. 

As a byproduct, our results allow one to determine the best base for towers (see definition in the next section) built for an FDL system. When used as the base for a tower, a choice of any element from a Markov partition of a FDL system ensures exponential decay of the first recurrence probabilities to this base. Our results allow one to chose base(s) with the fastest decay of the first recurrence probabilities. It gives hope that the theory of dynamical systems will be able to be developed to such stage when it would be possible to find numerical estimates of various exponential rates of convergence rather than dealing only with qualitative statements like that a certain rate is exponential.

 It is naive to expect that the most broad and important class of nonuniformly hyperbolic dynamical systems \cite{BP} will have the same properties as the FDL-systems. However, numerical experiments with dispersing billiards \cite{BV} demonstrate that there exist time intervals of finite lengths with hierarchies somewhat similar to the ones in the FDL-systems. Surely, one should expect that for general nonuniformly hyperbolic systems \cite{BP} there will be more intervals with alternating hierarchies of the first hitting probabilities than for FDL systems.  Actually, this was shown by some numerical results in \cite{BV}. To understand better what is going on, it is necessary to analyze some class of hyperbolic dynamical systems with distortion, i.e. with less uniform hyperbolicity than in the FDL-systems.

We also present in this paper an algorithm which allows accelerate escape from the phase spaces of strongly chaotic dynamical systems. This algorithm readily follows from the results of the present paper. This algorithm can be applied to real systems, particularly to atomic billiards \cite{FKCD,MHCR}. In a nutshell, it says the following: make a hole in a certain (optimal!) subset of the phase space and keep it open till a certain moment of time when this subset ceases to be optimal. Then close ("patch") this hole, and make a new hole in another subset which has become an optimal sink at this moment of time. The process continues by switching to other holes as they become optimal.

The structure of the paper is as follows. In the next section we provide necessary definitions and formulate main results. We also present there an algorithm which allows accelerate the process of escape from a FDL-system. In section 3 we introduce some notations and present several preliminary results. Section 4 provides a proof of the main results when considering subsets of phase space having the same measure, under a technical assumption whose proof is relegated to the appendix.  Also included in section 4 is a simple example demonstrating why just two time intervals with different hierarchies of the first hitting probabilities may exist. This main result is surprising and therefore such demonstration is helpful for understanding (and breaking up) long and formal proofs. Section 5 contains the proof of the main results for subsets of the phase space with unequal measures. In section 6 we briefly present a few results of computer simulations for the length of the short times interval. The last section 7 contains some concluding remarks.

\section{Definitions and Main Results}
Let $T: M \rightarrow M$ be a uniformly hyperbolic dynamical system preserving Borel probability measure $\mu$.  The following definition \cite{B} singles out a class of dynamical systems analogous to the independent, identically distributed (IID) random variables with uniform invariant distributions on their (finite!) state spaces.  Classical examples of such stochastic systems are fair coins and dices, hence the corresponding dynamical systems are called fair dice like (FDL) \cite{B}.

\begin{definition} \label{def FDL}
A uniformly hyperbolic dynamical system preserving Borel probability measure $\mu$ is called fair dice like or FDL if there exists a finite Markov partition $\xi$ of its phase space $M$ such that for any integers $m$ and $j_i$, $1 \leq j_i \leq q$ one has $\mu \left( C_\xi^{(j_0)} \cap T^{-1} C_\xi^{(j_1)} \cap \dots \cap T^{-m+1}C_\xi^{(j_{m-1})} \right) = \frac{1}{q^m}$ where $q$ is the number of elements in the partition $\xi$ and $C_\xi^{(j)}$ is element number $j$ of $\xi$.
\end{definition}
Therefore a FDL-system is a full Bernoulli shift with equal probabilities $\frac{1}{q}$ of states.
In what follows we will call a Markov partition in the definition of the FDL systems a basic Markov partition of the FDL-system under consideration.  We will be interested in such partitions of the phase space of FDL systems which are refinements of the basic Markov partition $\xi$ featured in the definition of FDL systems. We will say that the kth order refinement of the partition $\xi$ is the partition generated by the intersection of all elements of the partitions $T^{-i}\xi$ where $i$ varies between 1 and $k-1$. It is easy to see that this refinement has $q^k$ elements coded by the words of the length $k$. Clearly any refinement of a Markov partition is also a Markov partition. Therefore in what follows we often refer to refinements of the basic Markov partition as to Markov partitions.

\begin{example}
Let $Tx=qx$ (mod 1) where $x \in M = [0,1]$ and $q \geq 2$ is an integer, with $\mu$ the Lebesgue measure. The corresponding basic Markov partition is the one into equal intervals $[\frac{i}{q},\frac{i+1}{q}]$, $i=0, 1, \dots, q-1$.
\end{example}

\begin{example}

Consider the tent map $Tx$ = $2x$ if $0<x<1/2$ and $Tx$ = $2(1-x)$ if $1/2<x<1$ of the unit interval with Lebesgue measure. To see that it is a FDL-system take the same Markov partition into the intervals (0,1/2) and (1/2,1) as for the doubling map in the introduction.    
\end{example}

\begin{example}

Take now the von Neumann-Ulam map of the unit interval where $Tx=4x(1-x)$. This map preserves the measure $\mu$ with density $\frac{1}{\pi  \sqrt{x(1-x)}}$. The von Neumann-Ulam map is metrically conjugate to the tent map via the transformation $y=\sin^2(\frac{\pi x}{2})$. To see that it is also FDL just take the same Markov partition as for the tent and doubling maps.
\end{example}

\begin{example}
To see that FDL-systems could be high-dimensional as well consider the baker's map of the unit square, where $(x,y)\rightarrow(2x, y/2)$ if $0<x<1/2$ or $(x,y)\rightarrow(2x-1, y/2 + 1/2)$ if $1/2<x<1$. This map preserves Lebesgue measure (area). To see that baker's map is a FDL-system just take the Markov partition of the unit square into the strips $0<x<1/2$ and $1/2<x<1$
.     
\end{example}

Let $\Omega$ denote a finite alphabet of size $q \geq 2$. We will call any finite sequence composed of characters from the alphabet $\Omega$ a string or a word. For convenience both names will be used in what follows without ambiguity. For a fixed string $w = w_k \dots w_1$, $w_i \in \Omega$ let $a_w(n)$ denote the number of strings of length $n$ which do not contain $w$ as a substring of consecutive characters.  The survival probability for a subset of phase space coded by the string $w$ is then $\hat{P}_w(n) = \frac{a_w(n)}{q^n}$.

Denote $h_w(n) = q a_w(n-1)-a_w(n)$ for $n \geq k$.  It is easy to see that $h_w(n)$ equals the number of strings which contain the word $w$ as their last $k$ characters and do not have $w$ as a substring of $k$ consecutive characters in any other place.  Therefore $\frac{h_w(n+k)}{q^{n+k}}$ is the first hitting probability $P_w(n)$ of the word $w$ at the moment $n$.  

J. Conway suggested the notion of autocorrelation of strings (see \cite{GO}).  Consider any finite alphabet and denote by $|w|$ the length of the word $w$.  Let $|w|=k$.  Then the autocorrelation cor$(w)$ of the string $w$ is a binary sequence $b_k b_{k-1} \dots b_1$ where $b_i=1$ if $w_j = w_{k-i+j}$ for $j=1, \dots, i$, that is, if there is an overlap of size $i$ between the word $w$ and its shift to the right on $k-i$ characters.  For example, suppose that $w=10100101$ in a two symbols (characters) 0 and 1 alphabet .  Then cor$(w)=10000101$.

We can compare (values of) autocorrelations by considering them as numbers written in base 2. For instance the sequence 101 becomes the number 5.
Observe that the autocorrelation of a word $w$ is completely defined by its internal periodicities. Indeed all digits which equal to one in cor$(w)$ are at the positions corresponding to these internal periods \cite{GO,BY}.

Let $k=|w|$, $k'=|w'|$, and denote $h_w(n)=h(n)$ and $h_{w'}(n) = h'(n)$.  We define
\begin{align*}
s_w = \textrm{max}_{1 \leq j \leq k-1} \{ j: b_j=1 \}
\end{align*}
whenever this maximum exists and we let $s_w=0$ otherwise.  We will always denote $s=s_w$ and $s'=s_{w'}$.  In what follows we will generally denote any quantity or function that depends on $w'$ by a superscript $'$.

It follows from ergodicity that $\sum_{n=1}^\infty P_w(n) = 1 = \sum_{n=1}^\infty P_{w'}(n)$.  Therefore if $P_w(m)-P_{w'}(m) < 0$ for at least one $m$, there must be at least one $n$ for which $P_w(n)-P_{w'}(n) > 0$. Theorems 2.1 and 2.2 establish a surprising fact that for the FDL-systems there is only one $n$ for which the quantity $P_w(n)-P_{w'}(n)$ changes from being negative or zero to positive.

\begin{theorem}
Consider an FDL-system. Let $w$ and $w'$ be words coding some elements of (possibly different) refinements of the basic Markov partition such that cor$(w)>\textrm{cor}(w')$.  Then there exists an $N > k$ such that $h(n)-h'(n) \le 0$ for $n<N$, and $h(n)-h'(n)>0$ for $n>N$.
\end{theorem}

Observe that $2^{k-1} \leq \textrm{cor}(w) \leq 2^k-1$.  Therefore the assumption cor$(w)>\textrm{cor}(w')$ implies $k \ge k'$. (We also note that Eriksson's conjecture \cite {E,M} in discrete mathematics is a simple corollary of Theorem 2.1).

One may naturally expect that two discrete curves of survival probabilities intersect infinitely many times unless they coincide. (One gets the simplest example of two words of the same length with identical curves of the first hitting probabilities when all zeros in the first word are substituted by ones and all ones by zeros).

\begin{proposition}
Consider a FDL-system.If two words have the same length and equal autocorrelations then all first hitting probabilities for subsets coded by these words are equal to each other at any moment of time.
\end{proposition}


A proof of this proposition immediately follows from the definition of autocorrelations of words. Indeed the first hitting probability for any subset $A$ of the phase space equals its measure until the moment of time equal the minimal period of all periodic orbits which intersect $A$. At this moment the first hitting probability decreases by jumping to a smaller value. Such jumps occur at any moments of time corresponding to periods of periodic orbits intersecting $A$. It immediately follows from the definition of autocorrelation of words that any two sets coded by words with equal autocorrelations intersect only with such periodic orbits which have the same periods. Moreover the corresponding  jumps (decreases) of the first hitting probabilities which occur at the same moments are equal each other for the FDL-systems because all elements have the same measure.
Therefore the sequences of the first hitting probabilities for subsets of the phase space coded by the words with equal autocorrelations coincide. 

Consider the points $(n,P_w(n))$ on the plane where $n>0$ are integers. We get the first hitting probabilities curve for $w$
by connecting a point $(n,P_w(n))$ to $(n+1,P_w(n+1)$ by the straight segments for all $n$.
The next theorem establishes that nonidentical first hitting probabilities curves intersect only once.

\begin{theorem} \label{prob version}
With $N$ as given in Theorem 2.1 and under the same conditions,  there is an $N>k$ such that $P_w(n)-P_{w'}(n) \le 0$ for $n<N$, and $P_w(n) - P_{w'}(n) > 0$ for $n>N$.
\end{theorem}

According to Theorems 2.1 and 2.2, for two words with different lengths the corresponding first hitting probabilities curves intersect only at one point. This point divides the positive semi-line into a finite short times interval and an infinite long time interval. Before the moment of intersection it is less likely to hit the smaller subset of phase space (coded by the longer word) for the first time, and after the intersection it is less likely to hit the larger subset for the first time. For two elements of the same Markov partition (which have the same measure) there is also a short initial interval where the two corresponding first hitting probability curves coincide (unless these two curves coincide forever). The length of this initial interval does not exceed the (common) length of the code-words for elements of the Markov partition. After this interval there is a short times interval where it is more likely to visit one  (say the first) element of the Markov partition for the first time than the other one. The last interval is an infinite long times interval where it is more likely at any moment to visit for the first time the other one (second) element.

Take now all elements of a Markov partition. They have equal measures because we are dealing with FDL systems. Then there is initial time interval of the length equal the (same) length of words coding elements of this refinement of a basic Markov partition. After the initial interval comes a finite interval of short times where there is hierarchy of the first hitting probability curves. Then comes intermediate interval where (all!)  curves intersect. And finally there is infinite interval where there is a hierarchy of the first hitting probabilities curves which is opposite to the one in the short times interval. Therefore finite time predictions of dynamics are possible in the short times interval and in the last infinite long times interval.

\begin{figure}[htp] 
\centering
\includegraphics[width=10cm]{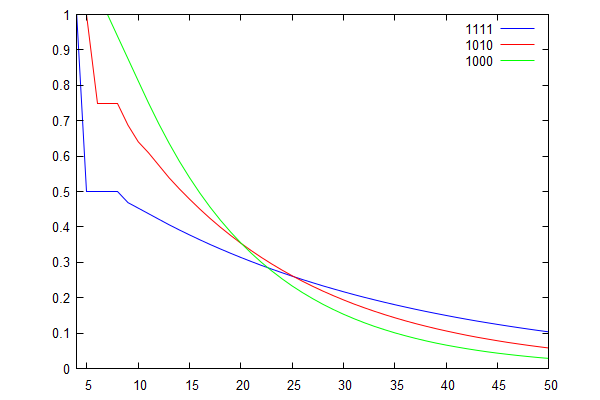} 
\caption{First hitting curves for the doubling map.}
\label{fig}
\end{figure}

Figure \ref{fig} illustrates the statement of Theorem \ref{prob version}. It depicts the first hitting probability curves for three different subsets within the domain of the doubling map.  These three subsets are encoded by the words 1111, 1010, and 1000.  Each subset has autocorrelation equal to the word that encodes it.

The next statement provides a lower bound on the length of the short times interval.

\begin{theorem}
Under the same conditions as Theorem 1 and with $N$ as defined there, if $k=k'$ and $s=s'$ then $N \ge 4k$.  If $k=k'$ and $s>s'$ then $N \ge 3k-s$.  If $k>k'$ then $N>k+1$.
\end{theorem}

Let a word $w$ correspond to a subset $A_w$ of some ergodic dynamical system. Because $\mu(A_w)>0$, almost all orbits return to the set $A_w$. Construct now a tower with base (zero floor) $A_w$. The $nth$ floor of a tower consists of all points of the set $T^{n}A_w$ which did not return to $A_w$ within the first $n$ iterations of $T$.   Denote by $R_{A_w}(n)$ probability of the first return to $A_w$ at the moment $n$. Let $P_{A_w}(n)$ be the first hitting (first passage) probability corresponding to the measure $\mu$.

\begin{definition}
Consider an ergodic dynamical system and choose two subsets $A$ and $B$ of positive measure.  We say that tower $Q_A$ with base $A$ is better than tower $Q_B$ with base $B$ if there exists $n^*$ such that $\sum_{n>n^*}R_A(n) < \sum_{n>n^*}R_B(n)$ for all $n>n^*$.\\  Let $\hat{\xi}$ be a refinement of the Markov partition $\xi$.  We say that an element $C_{\hat{\xi}}$ if the partition $\hat{\xi}$ is an optimal base for a tower out of all elements of $\hat{\xi}$ if there is no tower better than $Q_{C_{\hat{\xi}}}$.
\end{definition}

It is well known that the first hitting probability $P_(A_w)(n)$ of $A_w$ at the moment $n$ equals 
\begin{align*}
&P_{A_w}(n) = \sum_{m > n} R_{A_w}(m).
\end{align*}

 For a given refinement of the basic Markov  partition it is generally possible to have several optimal bases with equivalent towers built over them.  In view of above relation between the first hitting and the first return probabilities the following statement about an optimal base for a tower is an immediate corollary of Theorems 2.1 and 2.2.

\begin{theorem}
Consider an FDL-system. Then for any refinement of a basic Markov partition $\xi$ there exists an optimal tower with base from this refinement such that no other of its elements yields a tower better than this one.
\end{theorem}

 It is well known that for strongly chaotic (hyperbolic) dynamical systems periodic points are everywhere dense \cite{BP}. In particular it is true for FDL systems. Denote by $Per_{C_{\xi}}$ the minimal period out of all periodic orbits intersecting an element $C_{\xi}$ of some Markov partition $\xi$. A proof of Theorem 2.1 (see Section 4 and Appendix) implies the following lemma on periodic points and optimal towers.

\begin{lemma}
Consider an FDL-system. Let $\xi$ be any refinement of the basic Markov partition. An element  $C_{\hat{\xi}}$ such that the tower $Q_{C_{\hat{\xi}}}$ is optimal must have the maximum value of $Per_C$ out of all other elements of this Markov partition.
\end{lemma}

Generally an optimal tower for a given Markov partition is not unique, i.e. several elements can serve as optimal bases.

We conclude this section with presenting an algorithm which allows to speed up escape from open dynamical systems built out of FDL-systems. Consider a FDL-system and some refinement of the basic Markov partition. (Observe that a choice of such refinement determines the scale/precision at which we want to analyze dynamical system under study). Then make first a hole in such element $A$ of this refinement which has a minimal autocorrelation. At the moment when the first hitting probability curve corresponding to $A$ gets intersected by another such curve corresponding to another element $B$ "patch" the hole $A$ and make a hole in $B$. Keep a hole in $B$ until its first hitting probabilities curve gets intersected by a curve corresponding to subset (element of refinement) $C$. Then "patch" the "hole" $B$ and make "hole" in $C$. Continue this process by switching to a new hole after each intersection. It immediately follows from Theorems 2.1 and 2.2 that this process requires only a finite number of switches and it ensures the fastest escape from FDL-systems (for holes of a given size defined by the order of refinement of the basic Markov partition).

\section{Some Results on Pattern Avoidance.}

We establish the convention that $b_0=1$ for every word $w$, ie. each word is augmented by the last symbol (digit) $1$. The purpose of this convention is to simplify statements like the following, which without this convention do not make sense when $(k-s) | k$, for example.  By the definition of cor$(w)$
\begin{align} \label{cor per}
\begin{split}
&\textrm{if } b_j=1 \textrm{ for some } j \in \{0, 1, \dots, k-1 \},\\
&\textrm{then } b_{k-(k-j)t=1} \textrm{ for all } t \in \{1, \dots, \lfloor \frac{k}{k-j} \rfloor\},
\end{split}
\end{align}
(see (\cite{M})).

Given $i$ such that $b_i=1$, let $[i] = \textrm{max } \{j: b_j=1 \textrm{ and } i=k-t(k-j) \textrm{ for some } 1 \le t \le \lfloor \frac{k}{k-j} \rfloor \}$.  In light of the relation (\ref{cor per}), it is natural to define the set $I = \{[i]: b_i=1\}$.

We will need to distinguish a few other digits of the autocorrelation in addition to $s$.  Let
\begin{align*}
&d = \textrm{min}_{j \in I} \{ j: b_j=1 \}\\
&r = \textrm{max}_{j \in I} \{ j: b_j=1 \textrm{ and } b'_j=0 \}
\end{align*}
whenever they exist.

The largest member of $I$ is always $s$.  An effect of Propositions \ref{p1} and \ref{p2} below is that $I \subset \{1, \dots, k-s-2\} \cup \{s\}$.  A further consequence of Proposition \ref{p1} is that the only member $i$ of $I$ for which $|\{j: [j]=i\}| > 1$ is $s$, hence we define $S = \{i: [i]=s \}$.

Let $H_w(n)$ be the number of strings which end with $w$, begin with $w$, and which do not contain $w$ as a substring of $k$ consecutive characters in any other place.  For $n > k$ it is easy to see that $H(n)=qh(n-1)-h(n)$.  The probability of first returning to the "hole" given by $w$ is $\frac{H_w(n)}{q^n}$.

While $H_w(n)>0$ for $n>2k$, $H_w(2k)=0$ if and only if there is an $i$ for which $b_i=b_{k-i}=1$.  It is easy to see that the condition $b_i=b_{k-i}=1$ implies $b_{k-s}=1$, and this in turn can be used to prove that $(k-s) | k$.  We can thus evaluate $H_w(n)$ for $n \leq 2k$ as follows.
\begin{align} \label{H vals}
&H_w(n) = \begin{cases} 0 & \textrm{ if } n<k,\\ -1 & \textrm{ if } n=k,\\ 1 & \textrm{ if } n=2k-i \textrm{ for some } i \in I,\\0 & \textrm{ otherwise if } n<2k\\ 0 & \textrm{ if } n=2k \textrm{ and } (k-s) | k\\ 1 & \textrm{ if } n=2k \textrm{ and } (k-s) \nmid k. \end{cases}
\end{align}

It was proved in \cite{M} that
\begin{align} \label{h ineq}
h_w(n) \geq \begin{cases} (q-1) \sum_{t=1}^{k-s} h_w(n-t) & \textrm{if } 0 < s < k,\\ (q-1) \sum_{t=1}^{k-1} h_w(n-1) & \textrm{if } s=0, \end{cases}
\end{align}
and
\begin{align} \label{h rec}
h_w(n) = q h_w(n-1)-h_w(n-k)+\sum_{t=1}^{k-1} b_t H_w(n-k+t).
\end{align}
The latter formula is derived from the following relation \cite{GO}.
\begin{align} \label{original h}
h_w(n) = \sum_{t=1}^k b_i H_w(n+t).
\end{align} 

It is easy to see that $H(n) \le h(n-k)$ for $n>k$, and we will prove below that $(q-1) h(n-k-1) \le H(n)$ for $n > k+1$. This result is the content of Corollary \ref{stuff}.

\section{A Proof of Theorems 2.1 and 2.2 for the Case $\mathbf{k=k'}$.}

We prove in this section several technical results which will be used to deduce Theorems 1 and 2.The corresponding proofs are rather long and formal.It is often difficult to understand why really a statement is true although formally it is justified.What is actually a "mechanism" which ensures that a claim is correct? It is very important to have such idea especially for statements answering questions of a new type (rather than for the incremental ones). Therefore we start with presenting a simple example to demonstrate the process (mechanism) which ensures that the first hitting probability curves for FDL-systems have not more than one intersection.

Assume for simplicity that Markov partition in the definition of FDL-systems has just two elements labeled by the symbols 0 and 1. Consider now the second refinement of this Markov partition. This refinement is also a Markov partition with eight elements. Pick for instance the elements coded as (101) and (001). Clearly $cor(101)$=101=5> is greater than $cor(001)$=100=4. The reason for this inequality is that a periodic point with minimum period which belongs to the element coded by (101) has period two while a point with minimum period in the element coded by (001) has period three. Therefore a number of strings
of length $n$ which do not contain (101) will be larger than a number of strings which do not contain (001) for all $n>4$ \cite{GO,BY}. Increase now the length of all strings by one. Then a number of strings will double and some new strings appear which are ended by one or other of our two words. Such strings will be excluded from a future consideration because they contribute to the first hitting probability at this very moment for the corresponding word. Observe though that in case of the word (101) all such strings are ending by one and therefore more strings of the length $n+1$ will remain which have the last digit $0$.
The situation is the opposite for the word (001). Therefore more than half of the remaining strings of the length $n+1$ will generate the word (001) in two steps (i.e. among the strings of the length $n+3$). On another hand less than half of all survived string at the moment $n+1$ will generate strings of the length $n+3$ ending by (101). Such process will continue forever and therefore the first hitting probability curve for (101) after intersecting the one for (001) will always remain above it. 
The words of arbitrarily long lengths may contain many internal periodicities. In fact, periodic orbits are everywhere dense in a phase space of any FDL-system. Therefore each element of any refinement of a Markov partition contains infinitely many periodic points. 
This is the reason why the proofs for a general case become long and include consideration of many different cases. 

We turn now to the formal proofs of Theorems 1-2.  Observe at first that Theorem 2 is equivalent to the claim that an $N>k$ exists such that $h(n)-q^{k-k'}h'(n-k+k') \le 0$ for $n<N$ and $h(n)-q^{k-k'}h'(n-k+k') > 0$ for $n \ge N$.

For any $i \in I$ let $T_w(i)= \textrm{max} \{ t>0: w_{k-i} \dots w_{k-i-t+1} = w_{k-j} \dots w_{k-j-t+1} \textrm{ for some } j \in I, j>i \}$, with the convention that if the latter set is empty then $T_w(i)=0$.  Again we will often denote $T(i)=T_w(i)$ when $w$ is fixed.

\begin{example}
Consider the word $HTHTHHHTHTH$ over the alphabet $\Omega = \{ H, T \}$.  Then cor$(w)=10000010101$ and $I = \{ 5,3,1 \}$.  In this example $k=11$ and $w_{k-5} \dots w_1 = HHTHTH$, $w_{k-3} \dots w_1 = THHHTHTH$, and $w_{k-1} \dots w_1 = THTHHHTHTH$.  Then $T(5)=T(3)=0$ and $T(1)=2$ since the first two letters of $w_{k-3} \dots w_1$ agree with the first two letters of $w_{k-1} \dots w_1$.  Note that for the word $HTHTHTHTHTH$, cor$(w)=10101010101$ and $I=\{9\}$.  None of $5,3$ or $1$ are in $I$ because $5=11-3(k-s)$, $3=11-4(k-s)$, and $1=11-5(k-s)$ where $\lfloor \frac{k}{k-s} \rfloor = 5$.
\end{example}

\begin{proposition} \label{p1}
Let $i \in I-\{s\}$.  Then $i+T(i) < k-s$.
\end{proposition}

\begin{proof}
For any $i$ such that $b_i=1$ and for any $t$ satisfying $k-i \le t \le k$ one has $w_t \dots w_{t-(k-i)+1} = w_{t-l(k-i)} \dots w_{t-(l+1)(k-i)+1}$ for $0 \le l \le \lfloor \frac{t}{k-i} \rfloor -1$.  Further, if $e=t-\lfloor \frac{t}{k-i} \rfloor (k-i)$ and $e>0$ then $w_t \dots w_{t-e+1} = w_e \dots w_1$.  This is a consequence of the structure of the correlation function as described by (\ref{cor per}).  Therefore when $b_i=1$ we will say that $w$ contains a $k-i$ period.

Let $i \in I$.  Suppose first that $T(i)>0$ and for a contradiction suppose that $i+T(i) \geq k-s$.  Let $i' \in I$, $i'>i$ be such that $w_{k-i} \dots w_{k-i-T(i)+1} = w_{k-i'} \dots w_{k-i'-T(i)+1}$.

Since $b_{i'}=1$, $w_{i'} \dots w_1 = w_k \dots w_{k-i'+1}$ which implies $w_{k-(i'-i)} \dots w_{k-i'+1} = w_{i'-(i'-i)} \dots w_1 = w_k \dots w_{k-i+1}$.  Since $b_i=1$, similarly one has 
\begin{align} \label{000}
    &w_{k-(i'-i)} \dots w_{k-i'+1} = w_k \dots w_{k-i+1}. 
\end{align}
Since $i-(k-s) \geq -T(i)$ we have $w_{k-i'} \dots w_{k-i'+i-(k-s)+1} = w_{k-i} \dots w_{k-i+i-(k-s)+1} = w_{k-i} \dots w_{s+1}$.  Therefore
\begin{align*}
    &w_{k-(i'-i)} \dots w_{k-(i'-i)-(k-s)+1}=w_k \dots w_{k-i+1} w_{k-i'} \dots w_{k-(i'-i)-(k-s)+1}\\
    &=w_k \dots w_{k-i+1} w_{k-i} \dots w_{s+1} = w_k \dots w_{s+1},
\end{align*}
where we have used (\ref{000}) in the first equality.  

Since $w$ contains a $k-s$ period,
\begin{align} \label{k-s periodic}
\begin{split}
    &w_{k-(i'-i)-l(k-s)} \dots w_{k-(i'-i)-(l+1)(k-s)+1}\\
    &=w_{k-(i'-i)} \dots w_{k-(i'-i)-(k-s)+1}= w_k \dots w_{s+1}\\
    &=w_{k-l(k-s)} \dots w_{k-(l+1)(k-s)+1}
\end{split}
\end{align}
for every $0 \leq l \leq \lfloor \frac{k-i}{k-s} \rfloor -1$.  Further, for $e = k-(i'-i)-\lfloor \frac{k-(i'-i)}{k-s} \rfloor (k-s)$, if $e>0$ one has 
\begin{align} \label{k-s periodic remainder}
\begin{split}
w_e \dots w_1 &= w_{k-(i'-i)} \dots w_{k-(i'-i)-e+1}\\
&=w_{k-\lfloor \frac{k-(i'-i)}{k-s} \rfloor (k-s)} \dots w_{i'-i+1}
\end{split}
\end{align}
Together, \ref{k-s periodic} and \ref{k-s periodic remainder} imply that $w_{k-(i'-i)} \dots w_1 = w_k \dots w_{i'-i+1}$, or that $b_{k-(i'-i)}=1$.\\

Our goal now is to show that there is some index $i^*$ such that $b_{i^*}=1$, $i^*>i$, and $i=k-l(k-i^*)$ for some $l>0$.  Doing this would contradict the fact that $i \in I$.  Let $d_0=i'-i$.  We will construct a strictly decreasing sequence $\{ d_n \}_{n=0}^N$ of positive integers such that $b_{k-d_n}=1$, $d_n=k-l_nd_{n-1}-i$ where $l_n$ is the unique positive integer such that $k-l_n d_{n-1}>i>k-(l_n+1)d_{n-1}$, and $i^*=k-d_N$ has the desired property.

If there is some $t$ for which $k-td_0 = i'$ then $k-(t+1)d_0 = i$, and we may take $N=0$.  Similarly $N=0$ if $k-td_0=i$ for some $t$.  Otherwise, there exists $t$ for which $i' > k-td_0 > i$.  Since $b_{k-d_0}=1$, the word $w$ contains a $d_0$ period.  With $d_1 = k-td_0-i$ it is easy to see that $b_{k-d_1}=1$ (a more detailed exposition for general $n$ is below).

For $n>1$ suppose that $d_{n-1}$ is already defined.  If $i \ne k-ld_{n-1}$ for some $l>0$ then $N = n-1$.  In addition one cannot have $k-l_n d_{n-1} = k-l_{n-1}d_{n-2}$ as this implies $i-(k-l_nd_{n-1})=i-(k-l_{n-1}d_{n-2}) = d_{n-1}$ hence $k-(l_n-1)d_{n-1} = i$, and again $N = n-1$.  Otherwise denote $\iota = k-l_{n-1}d_{n-2}$ and observe that there is some $l_n$ such that $\iota > k-l_n d_{n-1} > i$.  Since $w$ contains a $d_{n-1}$ period, we will show that $b_{k-l_nd_{n-1} - i}=b_{k-d_n}=1$.  Let $\delta=\iota-(k-l_nd_{n-1})$.  Observe that with this notation, $d_{n-1} = \delta+d_n$.

For any $0 \le l < \lfloor \frac{k-d_n}{d_{n-1}} \rfloor$ one has
\begin{align*}
    w_{k-d_n-ld_{n-1}} \dots w_{k-d_n-(l+1)d_{n-1}+1}&=w_{k-d_n-(l_n-1)d_{n-1}} \dots w_{k-d_n-l_nd_{n-1}+1}\\
    &=w_{\iota} \dots w_{i+1}\\
    &=w_k \dots w_{k-d_{n-1}}\\
    &=w_{k-ld_{n-1}} \dots w_{k-(l+1)d_{n-1}}.
\end{align*}
In the first equality we have used the $d_{n-1}$ periodicity of $w$, in the second we have used the fact that $k-d_n-(l_n-1)d_{n-1} = \iota$, in the third that $b_\iota=1$, and again in the fourth the $d_{n-1}$ periodicity of $w$.

Let $e=k-d_n-\lfloor \frac{k-d_n}{d_{n-1}} \rfloor d_{n-1}$.  If $e>0$ and $e<\delta$ then one has
\begin{align*}
    w_e \dots w_1 &=w_{k-l_nd_{n-1}-d_n} \dots w_{k-l_nd_{n-1}-d_n-e+1}\\
    &=w_{i} \dots w_{i-e+1}\\
    &=w_k \dots w_{k-e+1}\\
    &=w_{k-\lfloor \frac{k-d_n}{d_{n-1}} \rfloor d_{n-1}} \dots w_{d_n+1}.
\end{align*}
If $e>0$ and $e>\delta$ then
\begin{align*}
    w_e \dots w_1 &= w_i \dots w_{i-\delta+1} w_{i-\delta} \dots w_{i-e+1}\\
    &=w_k \dots w_{k-\delta+1} w_{k-\delta} \dots w_{k-e+1}\\
    &=w_k \dots w_{k-e+1}\\
    &=w_{k-\lfloor \frac{k-d_n}{d_{n-1}} \rfloor d_{n-1}} \dots w_{d_n+1}.
\end{align*}
One thus has $w_{k-d_n} \dots w_1 = w_k \dots w_{d_n+1}$ and $b_{k-d_n}=1$.  Since $d_n < d_{n-1}$ as long as $[i] \ne k-d_{n-1}$, the sequence $\{d_n\}$ is strictly decreasing.  Since $k-d_n$ is bounded below by 1, there must be some $n$ for which $[i]=k-d_n$, and we let $N=n$.\\

If $T(i)=0$, the proof is similar to what we have just done.  Supposing $i+T(i) = i > k-s$, there is some $t$ for which $k-t(k-s) > i > k-(t+1)(k-s)$, otherwise $[i]=s$ and $i \notin I$.  Since $w$ contains a $k-s$ period, one can show that $b_{k-t(k-s)-i}=1$.  Again we can construct a strictly increasing sequence of integers $\{ i_n \}_{n=0}^N$ such that $b_{i_n}=1$ and $[i]=i_N$.  We omit the proof due to its redundancy.
\end{proof}

\begin{corollary} \label{c0}
$\{i: b_i=1\} =\{i: [i]=s\} \cup \left( I-\{s\} \right)$.
\end{corollary}

\begin{proof}
    For $i \in I-\{s\}$ one has $i<s$ and $i<k-s$, hence $i<\frac{k}{2}$.  If $b_j=1$ and $j \ne k-t(k-s)$ for any $t$, then $j = k-l(k-i)$ for some $i \in I-\{s\}$ if and only if $l=1$ since $k-i>\frac{k}{2}$.  Thus either $j \in \{i: [i]=s\}$ or $j=i$ for some $i \in I-\{s\}$.
\end{proof}

\begin{corollary} \label{stuff}
$H(n) \ge (q-1)h(n-k-1)$ for $n>k+1$.
\end{corollary}

\begin{proof}
Observe that (\ref{H vals}) implies $H(n) \ge 0 = (q-1)h(n-k-1)$ for $k+1 < n \le 2k$.

Let $2k < n < 3k-s$.  For $1 \le i \le s$ one has $k+i < n-k+i < 2k-s+i \le 2k$, so by (\ref{H vals}) we have $H(n-k+i)=1$ if and only if $b_{3k-n-i}=1$ and $3k-n-i \in I$.  If $3k-n-i \in I-\{s\}$ then $3k-n-i < k-s$ by Proposition 1 and hence $b_{3k-n-i}=0$ for $n \le 2k+s-i$.  In particular, $b_{3k-n-i}=0$ when $n=2k+1$ and $i \in I-\{s\}$.  Thus $\sum_{i=1}^s b_i H(k-i+1) \le 1$ and $\sum_{i=1}^s b_i H(n-k+i) \le n-2k$.  Since $h(n)=q^{n-k}$ for $k \le n < k-s$ one has
\begin{align*}
    H(n) &= h(n-k) - \sum_{i=1}^s b_i H(n-k+i)\\
    &=qq^{n-2k-1} - \sum_{i=1}^s b_i H(n-k+i)\\
    &= (q-1)q^{n-2k-1} + \left( q^{n-2k-1} - \sum_{i=1}^s b_i H(n-k+i) \right)\\
    &\ge (q-1)q^{n-2k-1} + \left( q^{n-2k-1} - (n-2k) \right)\\
    &\ge (q-1)q^{n-2k-1} = (q-1)h(n-k-1).
\end{align*}

For $n \ge 3k-s$ observe that $\sum_{i \in S} H(n-k+i) \le \sum_{i=1}^k b_i H(n-2k+s+i) = h(n-2k+s)$.  Then one has
\begin{align*}
    &H(n) = h(n-k) - \sum_{i=1}^s b_i H(n-k+i)\\
    &=h(n-k) - \sum_{i \in S} H(n-k+i) - \sum_{i \in I-\{s\}} H(n-k+i)\\
    &\ge h(n-k) - h(n-2k+s) - \sum_{i \in I-\{s\}} h(n-2k+i)\\
    &\ge (q-1) \sum_{t=1}^{k-s} h(n-k-t) - h(n-2k+s) - \sum_{i=1}^{k-s-2} h(n-2k+s+i)\\
    &\ge (q-1) h(n-k-1)
\end{align*}
where we have used Corollary \ref{c0}.
\end{proof}

\begin{proposition} \label{p2}
Suppose that $s \neq k-1$.  Then either $b_{t(k-s)}=0$ for every $1 \leq t \leq \lfloor \frac{k}{k-s} \rfloor$ or $b_{t(k-s)-1}=0$ for every $1 \leq t \leq \lfloor \frac{k+1}{k-s} \rfloor$
\end{proposition}

\begin{proof}
We use the following two statements, the first of which is obvious from Proposition 1.
\begin{align} \label{Proposition -1}
&\textrm{If } b_{k-s}=1 \textrm{ then } [k-s]=s.
\end{align}
\begin{align} \label{Proposition -2}
\begin{split}
&\textrm{If } b_{k-s-1}=1 \textrm{ then } [k-s-1]=s.
\end{split}
\end{align}
We prove (\ref{Proposition -2}).  Suppose $b_{k-s-1}=1$.  If $b_{k-s}=1$ then $b_t=1$ for every $1 \leq t \leq k-s$ and so $w_t = w_k$ for every $1 \leq t \leq k-s$.  Since $w$ contains a $k-s$ period $w=\underbrace{w_k*\dots*w_k}_{\textrm{k times}}$ and $[k-s-1]=s$.

If $b_{k-s}=0$ then either $k-s-1=s$ and the result follows, or there is some $t>0$ such that $k-t(k-s)>k-s$.  Let $i=k-s-1$, $\iota=k-t(k-s)$, and $d = \iota-i$.  Since $b_{\iota}=1$ and $b_i=1$ it is easy to see that $w_{\iota} \dots w_1$ contains a $d$ period.  As a result $b_{\iota-l d}=1$ for every $1 \leq l \leq \lfloor \frac{\iota}{d} \rfloor$.  Let $L$ be such that $\iota-Ld > k-(k+1)(k-s)>0$ and $\iota-(L+1)d \leq k-(k+1)(k-s)$.  If $d \nmid (k-s)$ then $\iota-Ld - \big( k-(k+1)(k-s) \big) = d' < d$.  Since $b_{\iota-Ld}=1$ and $b_{k-(k+1)(k-s)}=1$ it must be that $w_\iota \dots w_1$ contains a $d'$ period, and hence $b_{\iota-d'}=1$.  Since $\iota-d'>\iota-d=k-s-1$, with $i'=[\iota-d']$ one has $[i'] \neq s$ and $[i']+T([i']) \geq k-s$, a contradiction to Proposition 1.  It follows that $d | (k-s)$ which implies that $w$ itself contains a $d$ period.  Then $b_{k-d}=1$ but since $d<k-s$ this contradicts the definition of $s$.

The following statement is a corollary of (\ref{Proposition -1}) and (\ref{Proposition -2}).
\begin{align*}
&\textrm{Suppose that } s \neq k-1.\textrm{  Then } I-\{s\} \subset \{1, \dots, k-s-2 \}.
\end{align*}

If $b_{t(k-s)}=1$ for some $t>1$ then by Proposition (\ref{p1}) it must be that $[t(k-s)]=s$ and hence there is some $l$ such that $k-l(k-s)=t(k-s)$, whence $(k-s) | k$.  According to (\ref{cor per}) it must be that $b_{k-s}=1$ as well.  Thus, if $b_{k-s}=0$ then $b_{t(k-s)}=0$ for every $t$.

If $b_{k-s}=1$ then $b_{k-s-1}=0$ as otherwise $s=k-1$.  If $b_{t(k-s)-1}=1$ for some $t>1$ then $[t(k-s)-1]=s$ and $k=m(k-s)-1$.  Since $b_{k-t(k-s)}=1$, one has $b_{(m-t)(k-s)-1}=1$ for every $1 \le t \le m-1$.  With $t=m-1$ this implies in particular that $b_{k-s-1}=1$, a contradiction. Thus $b_{t(k-s)-1}=0$ for every $t>1$.
\end{proof}

\begin{lemma} \label{l1}
If $s>0$, $1 \leq l \leq k-1$, and $n \geq 2k+l$ then $H(n) \geq (q-1) \sum_{t=1}^l H(n-t)$.
\end{lemma}

\begin{proof}
Rearranging relation (\ref{original h}) we obtain
\begin{align} \label{H rearranged}
&H(n) = h(n-k)-\sum_{t=1}^{k-1} b_t H(n-k+t) = h(n-k)-\sum_{i=1}^s b_i H(n-k+i).
\end{align}
Using (\ref{H rearranged}) one has
\begin{align} \label{H k-s}
\begin{split}
&H(n) - (q-1) \sum_{t=1}^l H(n-t) =h(n-k) - (q-1) \sum_{t=1}^l h(n-k-t)\\
&-\sum_{i=1}^s b_i H(n-k+i) + (q-1) \sum_{t=1}^l \sum_{i=1}^s b_i H(n-k+i-t)=\\
&h(n-k-l) - \sum_{t=0}^{l-1} H(n-k-t) \\
&- \sum_{i=1}^s b_i \left( H(n-k+i) - (q-1) \sum_{t=1}^l H(n-k+i-t) \right).
\end{split}
\end{align}

For any $i \in I-\{d\}$ let $\tilde{i} = \textrm{max } \{\iota < i: b_\iota = 1\}$.  Observe that $|i-\tilde{i}| \leq k-s$ since $b_{k-t(k-s)}=1$ always.  It follows that $\tilde{i}=i-\tau$ for some $1 \leq \tau \leq k-s$.  We have
\begin{align} \label{cancel}
\begin{split}
&\sum_{i=1}^s b_i \left( H(n-k+i) - (q-1) \sum_{t=1}^l H(n-k+i-t) \right) \leq\\
&\sum_{i \in I-\{d\}} \Big( H(n-k+i) - H(n-k+\tilde{i}) \Big) +H(n-k+d)-(q-1) \sum_{t=1}^l H(n-k+d-t)\\
&\le H(n-l) - \sum_{t=0}^{l-1} H(n-k-t).
\end{split}
\end{align}
where we have used the fact that $n \ge 2k+l$ to ensure that $n-k+i-t > k$ for every $i$, hence $H(n-k+i-t)>H(k)=-1$.  Combining (\ref{cancel}) and (\ref{H k-s}) one has
\begin{align} \label{H reduced}
\begin{split}
&H(n) - (q-1) \sum_{t=1}^l H(n-t) \geq\\
&h(n-k-l) - \sum_{t=0}^{l-1} H(n-k-t) - H(n-l) + \sum_{t=0}^{l-1} H(n-k-t) \geq\\
&h(n-k-l) - H(n-l) \geq 0.
\end{split}
\end{align}
\end{proof}

Let $\Delta(n)=h(n)-h'(n)$.

\begin{corollary} \label{c1}
Suppose $w'$ is such that cor$(w) \geq \textrm{cor}(w')$, $k=k'$, and $n \ge 3k$.  If $\Delta(n-t) \ge (q-1)\Delta(n-t-1)$ for $1 \leq t \leq k-1$ then $\Delta(n) \ge (q-1)\Delta(n-1)$.
\end{corollary}

\begin{proof}
There are three cases.  In the first $1<r<k-1$, in the second $r=k-1$, and in the third $r=1$.  In the first case, using (\ref{h rec}) one has

\begin{align} \label{D}
\begin{split}
&\Delta(n)=q \Delta(n-1)-\Delta(n-k)+\sum_{i=r+1}^{k-1} b_i [q \Delta(n-k+i-1)-\Delta(n-k+i) ]\\
&+\sum_{i=1}^r b_i H(n-k+i)-\sum_{i=1}^{r-1} b_i' H'(n-k+i)
\end{split}
\end{align}

Using the equality $H(n)-H'(n) = q\Delta(n-1)-\Delta(n)$ and applying Lemma 1 one has
\begin{align*}
&\sum_{i=1}^r b_i H(n-k+i)-\sum_{i=1}^{r-1} b'_i H'(n-k+i)\\
&\ge \sum_{i=1}^{r-1} H(n-k+i) - \sum_{i=1}^{r-1} H'(n-k+i)\\
&= \sum_{i=1}^{r-1} q \Delta(n-k+i-1) - \Delta(n-k+i) 
\end{align*}
The expression (\ref{D}) is thus bounded below by

\begin{align*}
&q \Delta(n-1)-\Delta(n-k) + \sum_{i=r+1}^{k-1} b_i [ q \Delta(n-k+i-1)-\Delta(n-k+i)\\
&+\sum_{i=1}^{r-1} [q \Delta(n-k+i-1)-\Delta(n-k+i)].
\end{align*}
By using the inductive assumption it is easy to see that $\Delta(n-1) + \sum_{i=r+1}^{k-1} b_i[ g \Delta(n-k+i-1)-\Delta(n-k+i) ] \geq (q-1) \Delta(n-k+r)$ and $\Delta(n-k+r) + \sum_{i=1}^{r-1} [ q \Delta(n-k+i-1)-\Delta(n-k+i) ] \geq (q-1) \Delta(n-k)$.  (For a more detailed explanation, see \cite{M}).  Applying both bounds, we have $\Delta(n) \geq (q-1) \Delta(n-1) - \Delta(n-k)+(q-1) \Delta(n-k) \geq (q-1) \Delta(n-1)$.\\

Suppose now that $r=k-1$, hence $s=r=k-1$ and $b_i=1$ $\forall i$.  We have
\begin{align} \label{r=k-1}
\begin{split}
&\Delta(n) = q \Delta(n-1) - \Delta(n-k) + H(n-1)+\sum_{t=1}^{k-2} H(n-k-i)-b'_i H'(n-k-i)\\
&\geq q \Delta(n-1) - \Delta(n-k) + H(n-1) + \sum_{t=1}^{k-2} [q \Delta(n-k+i-1)-\Delta(n-k+i)].
\end{split}
\end{align}
Noting that $H(n-1)>0$ since $n \geq 3k$, by subtracting $H'(n-1)$ from (\ref{r=k-1}) we obtain
\begin{align*}
\begin{split}
&\Delta(n) \geq (q-1) \Delta(n-1)-\Delta(n-k)\\
&+\Delta(n-1) + \sum_{t=1}^{k-2} [q \Delta(n-k+t-1)-\Delta(n-k+t)].
\end{split}
\end{align*}
Using $\Delta(n-1)>\Delta(n-2)$ as before we have that
\begin{align*}
&\Delta(n) > (q-1) \Delta(n-1)-\Delta(n-k)+(q-1) \Delta(n-k) \\
&\geq (q-1) \Delta(n-1).
\end{align*}

Finally suppose that $r=1$.  Then 
\begin{align*}
&\Delta(n) = q \Delta(n-1) - \Delta(n-k) + \sum_{t=2}^{k=1} b_t [q \Delta(n-k+t-1)-\Delta(n-k+t)] + H(n-k+1)\\
&\geq (q-1) \Delta(n-1) - \Delta(n-k) + (q-1) \Delta(n-k+1)\\
&\geq (q-1) \Delta(n-1).
\end{align*}
\end{proof}

\begin{corollary} \label{c2}
Let $n \ge 4k$.  Then $\sum_{t=0}^{k-d-1} H(n-k-t) \ge \sum_{i \in I} \sum_{t=1}^{k-i} b_t H(n-2k+i+t)$
\end{corollary}

\begin{proof}
For $i \in I- \{s\}$, applying Lemma 1 we have
\begin{align*}
    &\sum_{t=1}^{k-i} b_t H(n-2k+i+t) = \sum_{t=1}^s b_t H(n-2k+i+t)\\
    &\le H(n-2k+i+s+1).
\end{align*}

Note that $I-\{s\} \subset \{1, \dots, k-s-2 \}$.  If $b_{k-s-2}=1$ and $b_{k-s}=1$ or $b_{k-s-1}=1$ then $I=\{s\}$ and the statement of the lemma holds.  It suffices to assume that either $b_{k-s-2}=0$ or $b_{k-s}=b_{k-s-1}=0$.  

If $b_{k-s-2}=0$ then $s \ne k-1$ and for $i \in I$ one has $i+s+1 \le k-s-2$.  Since it always true that one of $b_{k-s}$ or $b_{k-s-1}=0$ when $s \ne k-1$ (see Propostion 2) one must have
\begin{align*}
-H(n-k)-H(n-k-1)+\sum_{t=1}^{k-s} b_t H(n-2k+s+t) \le 0.
\end{align*}
Since $i+s+1 \le k-s-2$ for $i \in I-\{s\}$ we thus have
\begin{align*}
    &-\sum_{t=0}^{k-d-1} H(n-k-t) + \sum_{i \in I} \sum_{t=1}^{k-i} b_t H(n-2k+i+t)\\
    &\le -H(n-k)-H(n-k-1) + \sum_{t=1}^{k-s} b_t H(n-2k+s+t)\\
    &\le 0.
\end{align*}

If both $b_{k-s}=0$ and $b_{k-s-1}=0$ then 
\begin{align*}
-H(n-k)+\sum_{t=1}^{k-s} b_t H(n-2k+s+t) \le 0.
\end{align*}
For $i \in I$ one has $i+s+1 \le k-s-1$ and
\begin{align*}
    &-\sum_{t=0}^{k-d-1} H(n-k-t) + \sum_{i \in I} \sum_{t=1}^{k-i} b_t H(n-2k+i+t)\\
    &\le -H(n-k) + \sum_{t=1}^{k-s} b_t H(n-2k+s+t)\\
    &\le 0.
\end{align*}

\end{proof}

\begin{corollary} \label{c3}
Let $n > 2k$.  Then $H(n) \ge \sum_{i \in I} H(n-k+i)$.
\end{corollary}

\begin{proof}
Let $2k<n<3k-s$.  For $i \ne s$ observe that $n-k+i < n-2k+s < k$, hence $H(n-k+i)=0$.  It follows that
\begin{align*}
    H(n)-\sum_{i \in I} H(n-k+i) = H(n)-H(n-k+s) \ge H(n)-1 \ge 0.
\end{align*}
Note that $|I| \le k-s$ and recall $s \le k-1$.  For $n \ge 3k-s$ one has
\begin{align*}
    H(n) - \sum_{i \in I} H(n-k+i) \ge H(n)-\sum_{i=1}^{k-s} H(n-i) \ge 0
\end{align*}
by application of Lemma 1.    
\end{proof}

\begin{corollary} \label{c4}
Let $n \ge 2k$.  Then $h(n) \ge \sum_{t=1}^{k-1} H(n+t)$.
\end{corollary}

\begin{proof}
Applying Lemma 1, one has
\begin{align*}
    &h(n) = \sum_{t=1}^k b_t H(n+t) \ge \sum_{t \in S \cup \{k\}} H(n+t) \ge \sum_{t=1}^{k-1} H(n+t).
\end{align*}
\end{proof}

Let $S = \{ \iota : [\iota]=s \}$.  If $\iota \in S-\{s\}$ then
\begin{align*}
&h(n-2k+\iota) - \sum_{j=1}^s b_j H(n-2k+\iota+j) - H\big(n-2k+s+\iota+(k-s)) = 0.
\end{align*}
It follows that
\begin{align*}
&\sum_{\iota \in S-\{s\}} \left( h(n-2k+\iota) - \sum_{j=1}^s b_j H(n-2k+\iota+j) \right)\\
&+h(n-2k+s) - \sum_{j=1}^s b_j H(n-2k+s+j)\\
&=\sum_{\iota \in S-\{s\}} \left( h(n-2k+\iota)-\sum_{j=1}^s b_j H(n-2k+\iota+j) - H\big(n-2k+s+\iota+(k-s) \big) \right)\\
&+h(n-2k+s)- \sum_{j=1}^{k-s} b_j H(n-2k+s+j).
\end{align*}

From this one has
\begin{align} \label{cancellation of non-prime periods}
\begin{split}
&\sum_{\iota=1}^{k-1} b_i H(n-k+\iota) = \sum_{\iota \in S} H(n-k+\iota) + \sum_{\iota \in I-\{s\}} H(n-k+\iota)=\\
&\sum_{\iota \in S} \left( h(n-2k+\iota) - \sum_{t=1}^{k-1} b_t H(n-2k+\iota+t) \right)+\\
&\sum_{\iota \in I-\{s\}} \left( h(n-2k+\iota) - \sum_{t=1}^{k-1} b_t H(n-2k+\iota+t) \right)=\\
&h(n-2k+s) - \sum_{t=1}^{k-s} H(n-2k+s+t) \\
&+\sum_{\iota \in I-\{s\}} \left( h(n-2k+\iota) - \sum_{t=1}^{k-\iota} b_t H(n-2k+\iota+t) \right).
\end{split}
\end{align}

\begin{lemma} \label{l2}
If $k=k'$ and $s \ne s'$ then $\Delta(n) \le 0$ for $n < 3k-s$.  If $s = s'$ then $\Delta(n) \le 0$ for $n < 4k$.
\end{lemma}

The proof of Lemma \ref{l2} when $s=s'$ is divided into two parts which constitute the appendices below.  We include the proof when $s \ne s'$ here.

\begin{proof}
We remark that $\Delta(n) \le 0$ for $n \le 2k$ and $\Delta(n) = 0$ for $n < 2k-r$ no matter the values of $s$ and $s'$.

Using (\ref{H rearranged}) one has 
\begin{align} \label{delta H}
\begin{split}
H(n)-H'(n) &= \Delta(n-k) - \sum_{i=1}^s b_i H(n-k+i)- b'_i H'(n-k+i).
\end{split}
\end{align}

Suppose that $r=s$.  There are two cases;  $b_t-b'_t=-1$ for $1 \le t \le s-1$, or $b_\tau-b'_\tau=0$ for some $\tau<s$.  In the first case note that $s'=s-1$, and observe that one must have $s \le k-s$.  Otherwise $s > s-(k-s) \ge 1$ and since $s-(k-s)=k-2(k-s)$ one has $b_{k-2(k-s)}-b'_{k-2(k-s)} \ge 0$, a contradiction.

From (\ref{delta H}) for $n<3k-s$ one has
\begin{align*}
H(n)-H'(n) &= -H(n-k+s) + \sum_{t=1}^{s-1} H'(n-k+t) \ge -H(n-k+s).
\end{align*}
It is thus easy to see that
\begin{align*}
&H(n)-H'(n) \ge \begin{cases} 0 & 2k < n < 3k-2s \\ -1 & 3k-2s \le n < 3k-s.\end{cases}
\end{align*}

Since $\Delta(2k) \le -1$, using the relation
\begin{align*}
\Delta(n) = q\Delta(n-1)-H(n)+H'(n)
\end{align*}
one has
\begin{align*}
&\Delta(n) \le \begin{cases} q \Delta(n-1) &  2k < n < 3k-2s \\q \Delta(n-1)+1 & 3k-2s \le n < 3k-s. \end{cases}
\end{align*}
It follows that $\Delta(n) \le 0$ for $n<3k-s$.\\

We now suppose that $b_\tau-b'_\tau=0$ for some $1<\tau<s$.  One has
\begin{align} \label{estimate delta 2k}
&\Delta(2k) \le -q^s+\sum_{t=1}^s q^{s-t} - q^{s-\tau} \le -2.
\end{align}
From (\ref{delta H}) one has
\begin{align*}
H(n)-H'(n) &\ge -\sum_{i=1}^s b_i H(n-k+i),
\end{align*}
from which we obtain the upper bound
\begin{align*}
&H(n)-H'(n) \ge -(n-2k).
\end{align*}
It follows that
\begin{align} \label{bound r=s 2}
&\Delta(n) \le q\Delta(n-1)+(n-2k).
\end{align}
The inequality (\ref{bound r=s 2}) together with (\ref{estimate delta 2k}) implies
\begin{align*}
\Delta(n) \le -2-(n-2k) \textrm{ for } 2k<n<3k-s
\end{align*}
and the result follows.
\end{proof}

\begin{lemma} \label{l3}
Let $k=k'$ and $N \ge 4k$.  Suppose $h(N) > h'(N)$ and that $h(n) \leq h'(n)$ for $n<N$.  Then $H(n)-H'(n) \le 0$ for $N \leq n \leq N+k$;  In particular $\Delta(n) > (q-1) \Delta(n-1)$ for $N \leq n \leq N+k$.
\end{lemma}

\begin{proof}
Suppose that $r=s$, and observe that this is equivalent to the condition $s>s'$.  By Lemma 4.2 we may assume that $N \ge 3k-s$.  From (\ref{H rearranged}), (\ref{cancellation of non-prime periods}), and (\ref{chain}) one has
\begin{align*}
H(n) &\le h(n-k)-\sum_{i \in S} H(n-k+i)\\
&=h(n-k)-h(n-2k+s)+\sum_{t=1}^{k-s} b_t H(n-2k+s+t)\\
&=(q-1)\sum_{t=1}^{k-s} h(n-k-t) - \sum_{t=0}^{k-s-1} H(n-k-t) + \sum_{t=1}^{k-s} b_t H(n-2k+s+t)\\
&\le (q-1)\sum_{t=1}^{k-s} h(n-k-t).
\end{align*}
To summarize,
\begin{align} \label{H upper 1}
&H(n) \le (q-1) \sum_{t=1}^{k-s} h(n-k-t).
\end{align}

For any $L>0$ one has
\begin{align} \label{chain}
    &h(n) = (q-1) \sum_{t=1}^L h(n-t) - \sum_{t=0}^{L-1} H(n-t).
\end{align}

Note that $s'<k-1$.  From Proposition \ref{p2} there exists $\iota \in \{ 0, 1 \}$ such that $b'_{t(k-s')-\iota}=0$ for every $1 \le t \le U$ where $U = \lfloor \frac{k+\iota}{k-s'} \rfloor$.  Using (\ref{h ineq}) and (\ref{chain}) one has
\begin{align*}
H'(n) &\ge h'(n-k) - \sum_{i=1}^{s'} b'_i h'(n-2k+i)\\
&= (q-1) \sum_{t=1}^k h'(n-k-t) - \sum_{t=0}^{k-1} H'(n-k-t) - \sum_{i=1}^{s'} b'_i h'(n-2k+i)\\
&\ge (q-1) \sum_{t=1}^{k-s'-1} h'(n-k-t) - \sum_{t=0}^{k-1} H'(n-k-t) \\
&+ \sum_{t=1}^U h'(n-2k-\iota+t(k-s'))+h'(n-2k)\\
&=(q-1) \sum_{t=1}^{k-s'-1} h'(n-k-t) +\left( h'(n-2k)-\sum_{l=1}^{k-s'-\iota} H'(n-2k+l) \right)\\
&+\sum_{t=1}^{U-1} \left( h'(n-2k-\iota+t(k-s')) - \sum_{l=1}^{k-s'} H'(n-2k-\iota+t(k-s')+l) \right)\\
&+ \left( h'(n-2k-\iota+U(k-s'))-\sum_{l=1}^{k+\iota-U(k-s')} H'(n-2k-\iota+U(k-s')+l) \right)\\
&\ge (q-1) \sum_{t=1}^{k-s'-1} h'(n-k-t).
\end{align*}
Thus
\begin{align} \label{H lower 1}
&H'(n) \ge (q-1)\sum_{t=1}^{k-s'-1} h'(n-k-t).
\end{align}

Using (\ref{H upper 1}) and (\ref{H lower 1}) it follows that
\begin{align*}
H(n)-H'(n) &\le (q-1) \sum_{t=1}^{k-s} h(n-k-t) - (q-1) \sum_{t=1}^{k-s'-1} h'(n-k-t) \le 0
\end{align*}
for $N \le n < N+k$.\\

Suppose that $r \neq s$, equivalently $s=s'$.  By Lemma \ref{l2} we may assume that $n \ge 4k$.  By use of Corollary \ref{c2} and equality (\ref{cancellation of non-prime periods}) one has
\begin{align*}
    &H(n) = h(n-k) - \sum_{i \in I} h(n-2k+i) + \sum_{i \in I} \sum_{t=1}^s b_t H(n-2k+i+t)\\
    &=(q-1) \sum_{t=1}^{k-d} h(n-k-t) - \sum_{t=0}^{k-d-1} H(n-k-t) - \sum_{i \in I-\{d\}} h(n-2k+i)\\
    &+\sum_{i \in I} \sum_{t=1}^s b_t H(n-2k+i+t)\\
    &\le (q-1) \sum_{t=1}^{k-d} h(n-k-t)- \sum_{i \in I-\{d\}} h(n-2k+i).
\end{align*}
It is easy to see that
\begin{align*}
    &H'(n) \ge h'(n-k) - \sum_{i \in I'} h'(n-2k+i)\\
    &=(q-1)\sum_{t=1}^{k-d'} h'(n-k-t) - \sum_{i \in I'-\{d'\}} h'(n-2k+i) - \sum_{t=0}^{k-d'-1} H'(n-k-t).
\end{align*}
It follows that
\begin{align} \label{this}
\begin{split}
    &H(n)-H'(n) \le (q-1) \sum_{t=1}^{k-d} h(n-k-t) - \sum_{i \in I-\{d\}} h(n-2k+i)\\
    &-(q-1) \sum_{t=1}^{k-d'} h'(n-k-t) + \sum_{i \in I'-\{d'\}} h'(n-2k+i) + \sum_{t=0}^{k-d'-1} H'(n-k-t).
\end{split}
\end{align}

Suppose $d' \le d$.  One has
\begin{align*}
    &H(n)-H'(n) \le -h'(n-2k+r) + \sum_{i=d+1}^{r-1} h(n-2k+i)+ \sum_{t=0}^{k-d'-1} H'(n-k-t)\\
    &\le -h'(n-2k+r)+\sum_{i=d+1}^{r-1} h'(n-2k+i) + \sum_{t=0}^{k-d'-1} H'(n-k-t)\\
    &\le -h'(n-2k+d) + \sum_{t=0}^{k-d'-1} H'(n-k-t) \le 0,
\end{align*}
using Corollary \ref{c4} and the fact that $h'(n-2k+r) - (q-1) \sum_{t=1}^{r-d-1} h'(n-2k+r-t) \ge (q-1) h'(n-2k+d)$ since $r-d-1 < k-s$.

Suppose $d' > d$.  If $r=d$ then $I= I' \cup \{d\}$.  Inequality (\ref{this}) becomes
\begin{align*}
&H(n)-H'(n) \leq (q-1) \sum_{t=k-d'+1}^{k-d} h(n-k-t) - h'(n-2k+d')+\\
&\sum_{t=0}^{k-d'} H'(n-k-t) \leq -h'(n-2k+d-1) + \sum_{t=0}^{k-d} H'(n-k-t) \leq 0,
\end{align*}
where we have used Corollary \ref{c4}.

Assuming now that $r>d'$ we add and subtract $h'(n-2k+d')$ to inequality (\ref{this}) and use corollary \ref{c4} to obtain
\begin{align*}
&H(n)-H'(n) \leq (q-1) \sum_{t=k-r}^{k-d} h(n-k-t) - h(n-2k+r)-\\
&(q-1) \sum_{t=k-r}^{k-d'} h'(n-k-t) + \sum_{t=r-1}^{d'} h'(n-2k+i) - h'(n-2k+d')+\\
&\sum_{t=0}^{k-d'-1} H'(n-k-t) \leq\\
&\sum_{t=k-d'+1}^{k-d} h(n-k-t) - h'(n-2k+r)+\sum_{t=k-r+1}^{k-d'} h(n-k-t) \leq 0.
\end{align*}

Finally, if $d'>r>d$ then $I' \subset I \cap \{k-1, \dots, r+1 \}$ and we have
\begin{align*}
    &H(n)-H'(n) \le (q-1) \sum_{t=k-d'+1}^{k-d} h(n-k-t) - h'(n-2k+d')\\
    &-h(n-2k+d)+\sum_{t=0}^{k-d'-1}H'(n-k-t) \le 0.
\end{align*}
\end{proof}

By combining statements proved in this section one can deduce Theorems 1 and 2 for the case $k=k'$.  Indeed, let $\Delta(n) = h(n)-h'(n)$.  Then $\Delta(n) \le 0$ for $n < 4k$, hence $N \ge 4k$.  According to Lemma \ref{l3} one has $\Delta(N+t) \ge q \Delta(N+t-1) \ge (q-1) \Delta(N+t-1)$ for $0 \le t \le k-1$, and by Corollary \ref{c1} this implies $\Delta(N+k) \ge (q-1) \Delta(N+k-1)$.  By a simple inductive argument, Corollary \ref{c1} then implies that $\Delta(n) \ge (q-1) \Delta(n-1)$ for any $n \ge N$.  Theorem 2 follows from Theorem 1 by observing that $P_w(n) = h(n+k)/q^{n+k}$ and likewise $P_{w'}(n) = h'(n+k)/q^{n+k}$. Finally Theorem 3 is an immediate consequence of Lemma \ref{l2}.  

\section{A Proof of Theorems 2.1 and 2.2 for the Case $\mathbf{k>k'}$}

\begin{lemma} \label{5.1}
Let $k>k'$.  Then $h(n)-q^{k-k'}h'(n-k+k') \le 0$ for $n \le k+1$.
\end{lemma}

\begin{proof}
Note that $h(n)-q^{k-k'}h'(n-k+k') = 0$ for $k-k' < n < k$ and $h(n)-q^{k-k'}h'(n-k+k') = 1-q^{k-k'} < 0$ when $n=k$.  When $n=k+1$, for any $0 \le m \le k-k'$ one has 
\begin{align*}
h(k+1)-q^mh'(k'+1) &\le q-q^{m+1}+q^{m}H'(k'+1)\\
&\le -q^{m+1}+2q^{m} \le 0.
\end{align*}
\end{proof}

\begin{lemma} \label{5.2}
Let $k>k'$ and $N > k+1$.  If $h(n) - q^{k-k'}h'(n-k+k') \le 0$ for $n<N$, then $H(n)-q^{k-k'}H'(n-k+k') \le 0$ for $N \leq n \leq N+k$.
\end{lemma}

\begin{proof}
For $0 \le m \le k-k'$ one has 
\begin{align*}
    H(n)-q^m H'(n-m) &\le h(n-k)-q^m(q-1)h'(n-k'-1-m)\\
    &\le h(n-k)-q^m h'(n-k-m),
\end{align*}
and the result follows.
\end{proof}

Note that in Lemma \ref{5.2} the inequality $H(n)-q^{k-k'}H'(n-k+k') \le 0$ is equivalent to $q \left( h(n-1)-q^{k-k'} h'(n-k+k'-1) \right) \le h(n)-q^{n-k+k'} h'(n-k+k')$.

\begin{lemma} \label{5.3}
Let $k>k'$ and $n \ge 2k$.  If $h(n-t)-q^{k-k'}h'(n-k+k'-t) \ge (q-1) \left( h(n-t-1)-q^{k-k'} h'(n-k+k'-t-1) \right)$ for $1 \leq t \leq k-1$ then $h(n)-q^{k-k'}h'(n-k+k'-t) \ge (q-1) \left( h(n-1)-q^{k-k'}h'(n-k+k') \right)$.
\end{lemma}

\begin{proof}
For any $0 \le m \le k-k'$ we have
\begin{align*}
h(n)-q^m h'(n-m)=&qh(n-1)-H(n)-q^{m+1}h'(n-m-1)+q^m H'(n-m)\\
=&q\left( h(n-1)-q^m h'(n-m-1) \right) +q^m H'(n-m)-H(n)\\
\ge &q\left( h(n-1)-q^m h'(n-m-1) \right)\\
&+q^m (q-1) h'(n-k'-m-1)-H(n)\\
\ge &q\left( h(n-1)-q^m h'(n-m-1) \right)\\
&+q^m h'(n-k'-m-1)-h(n-k)\\
\ge &q\left( h(n-1)-q^m h'(n-m-1) \right)+q^m h'(n-k-m)-h(n-k).
\end{align*}
Denote $K=k-k'$.  With $m=K$ we apply the inductive assumption to obtain
\begin{align*}
&h(n)-q^m h'(n-m)\\
\ge &q\left( h(n-1)-q^K h'(n-K-1) \right)+q^K h'(n-k-K)-h(n-k)\\
\ge &q\left( h(n-1)-q^K h'(n-K-1) \right)+q^K h'(n-1-K)-h(n-1)\\
=&(q-1)\left( h(n-1)-q^K h'(n-K-1) \right).
\end{align*}
\end{proof}

The lemmas of this section combine to prove the main theorems when $k>k'$ in the following way.  Let $\Delta(n) = h(n)-q^{k-k'}h'(n-k+k')$.  Then according to Lemma \ref{5.3} one has $\Delta(n) \le 0$ for $n \le k+1$ hence $N > k+1$, which is the statement of Theorem 3.  From Lemma \ref{5.2} one has $\Delta(N+t) \ge q \Delta(N+t-1) \ge (q-1) \Delta(N+t-1)$ for $0 \le t \le k-1$, and by Lemma \ref{5.3} this implies $\Delta(N+k) \ge (q-1) \Delta(N+k-1)$.  By a simple inductive argument, Lemma \ref{5.3} then implies that $\Delta(n) \ge (q-1) \Delta(n-1)$ for any $n \ge N$.  Dividing $\Delta(n)$ by $q^{n-k}$ then yields Theorem 2.

Observe that Lemmas \ref{5.1}, \ref{5.2}, and \ref{5.3} hold if everywhere in their statements we replace $k-k'$ with any $m$ satisfying $0 \le m \le k-k'$.  Theorem 1 is a consequence of setting $m=0$.

\section{Numerical Results on Lengths of Short and Intermediate Times Intervals}

If we think about the applicability of our results to finite time predictions of dynamics then one is led to the following key question: How long is the short times interval where there exists a hierarchy of the first hitting probabilities curves, within which predictions are possible for any moment of time? Another time interval where such predictions can be made is the last (third) infinite times interval. Therefore it is of great importance for applications to estimate the lengths of two finite intervals, the short times interval where finite time predictions are possible and the second, intermediate interval.

Clearly these lengths depend on $k$, i.e. the lengths of the words which are coding the subsets of phase space which we consider.

Theorem 2.3 gives a linear estimate on the length of the short times interval. However, numerical simulations show that the lengths of both of these intervals grow exponentially (asymptotically as the common length $k$ of the words increases).  If this is indeed the case then, practically speaking, only the short times interval is of interest because experiments and observations are usually not terribly long. In particular, if one is making observations about small subsets of the phase space, the short times interval may become very long, and thus covers the entire time of a reasonable (practically possible) experiment.

The following table presents the beginning and ending moments of the intermediate interval for the doubling map of $[0,1]$, i.e. the moments of time when the first and the last pairs of the first hitting probability curves intersect, respectively. Recall that the intermediate interval starts at the moment when the short times interval ends.
Notably, the length of the short times interval was in our computer experiments always larger than the length of the intermediate interval. It also appears that the ratio of lengths of these intervals converges to one in the limit when $k$ tends to infinity.

\begin{center}
 \begin{tabular}{||c c c||} 
 \hline 
 $k$ & Beginning of interval & End of interval \\ [0.5ex] 
 \hline\hline
 $4$ & 20 & 26 \\ 
 \hline
 $5$ & 37 & 52 \\
 \hline
 $6$ & 70 & 103 \\
 \hline 
 $7$ & 135 & 208 \\
 \hline
 $8$ & 264 & 415 \\[1ex] 
 \hline
\end{tabular}
\end{center}

\begin{figure}[htp] 
\centering
\includegraphics[width=10cm]{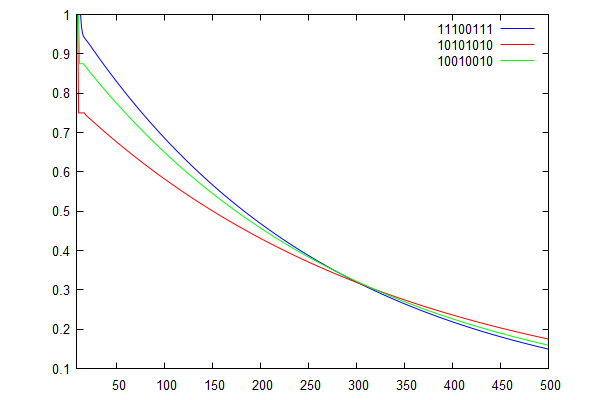} 
\caption{First hitting curves for the doubling map.}
\label{fig2}
\end{figure}

In figure \ref{fig2} we present three first hitting probability curves for the doubling map. These curves correspond to the words $10101010$, $10010010$, and $11100111$ of length eight. (Recall that Figure \ref{fig} presents similar results for words of length four). 

Recall that as the number of elements in a refinement of the Markov partition increases so does $k$ too, which is the length of the word coding each element. Therefore when we consider dynamics at finer scales, the length of the short times interval (on which predictions about the dynamics can be made) seems to grow exponentially. It is a very encouraging (although numerical) result which suggests that finite time predictions of dynamics could be made on very long time scales if we consider partition of the phase space with a sufficiently large number of elements.

\section{Concluding Remarks}
Our results demonstrate that interesting and important finite time predictions for the dynamics of systems with the strongest chaotic properties and for the most random stochastic processes are possible. They also indicate how such predictions can be practically made. Numerical simulations \cite{BV} suggest that some finite time predictions for nonuniformly hyperbolic systems are also possible.
However finite time predictions in this case will be not as simple as for FDL-systems which are the uniformly hyperbolic systems with the maximal possible uniformity. It seems natural to expect that for more general classes of chaotic dynamical systems there will be more than two time intervals with different hierarchies of the first hitting probabilities. Some (although too vague for formulating exact conjectures) indications of that can be extracted from computer simulations of dispersing billiards \cite{BV}  

Although the theory of finite time dynamics of chaotic systems is in infancy, it is rather clear what to do next and to which classes of dynamical systems these results should be generalized. 
For FDL-systems a remaining problem is to prove better estimates of the length of the short times interval.
A natural next step would be to consider IID-like but not FDL dynamical systems. Consider for instance a skewed tent map of the unit interval, i.e. $f(x)=ax$ if $0<x<1/a$ and $f(x)=ax/(1-a)$ if $1/a<x<1$, where $a>1$. This map is not an FDL-system (unless $a=2$) because the absolute values of the derivatives differ at different points of the phase space. (Therefore it is a system with distortion). Then one should follow a standard path in developing dynamical systems theory by trying to obtain finite time dynamics results for more and more non-uniformly hyperbolic dynamical systems.
For instance a natural question is whether it would be possible to find optimal Young towers \cite{LY} for some interesting classes of nonuniformly hyperbolic dynamical systems and estimate corresponding exponents in decay of the first recurrences probabilities. 
A significant problem is to develop relevant  mathematical approaches and techniques, more dynamical than combinatorial in spirit, to handle the questions arising in the studies of finite time dynamics.
It is also worthwhile to mention that in the Equilibrium Statistical Mechanics the main problem is about phase transitions, i.e. existence or non-existence of several equilibrium distributions. If there are no such transitions then usually problems of the Nonequilibrium Statistical Mechanics come under attack, i.e. what is a rate of approach to the equilibrium etc. The FDL-systems demonstrate that even when a system is in unique equilibrium (i.e. there are no phase transitions) still its evolution could be nonhomogeneous and have interesting properties pertaining to the transport in the phase space.  

\section*{Appendix 1. An Upper Bound for $\mathbf{\Delta(n)}$ on the Interval $\mathbf{2k < n \le 4k}$ when $\mathbf{s=s'}$}

Viewing $\Delta(n)$ as a function of the values $b_i$ and $b'_i$ for $i<r$, we will show that if $b_i=0$ then $\Delta(b_i+1, n)-\Delta(b_i,n) \le 0$.  One can also show that if $b_i=1$ then $\Delta(b_i-1,n)-\Delta(b_i,n) \ge 0$. In addition, if $b'_i = 0$ or $1$ then $\Delta(b_i'+1, n)-\Delta(b'_i, n) \ge 0$ or $\Delta(b'_i-1,n) - \Delta (b'_i,n) \le 0$, respectively. Because of the almost complete redundancy in all these calculations, we will only display the former case.

Let
\begin{align*}
    \tilde{b}_t = \begin{cases} 1 & b_t=1 \textrm{ and } t \in I\\ 0 & \textrm{ otherwise} \end{cases}
\end{align*}
and
\begin{align*}
    \delta (n)=  \sum_{t=1}^{3k-n-1} b_t \tilde{b}_{3k-n-t} - b'_t \tilde{b}'_{3k-n-t} .
\end{align*}
For $2k \le n<3k-s$ one has
\begin{align} \label{first app 1}
&H(n)-H'(n) = -\delta(n).
\end{align}
To see this, observe that a word of length $n$ beginning and ending with $w$ contains a copy of $w$ beginning in position $n-k+t$ if and only if $b_t=1$ and $H(t+(n-2k)+k)=1$. Therefore we have the equality $H(n) = q^{n-2k}-\sum_{t=1}^k b_t H(n-k+t)$. Since $n -k+t < 2k-s+t \le 2k$ we may apply the identity $H(n-k+t) = \tilde{b}_{3k-n-t}$, which proves (\ref{first app 1}).

We establish the convention that $\delta(n)=0$ when $n < 2k$. When $n=3k-s$ one has
\begin{align*}
H(n)-H'(n) &= -\delta(n) - \sum_{t=3k-n}^s b_t \left( H(n-k+t)- H'(n-k+t) \right) \\
&= -\delta(n)+\delta(n-k+s).
\end{align*}
We now prove inductively that
\begin{align*}
    H(n)-H'(n) = -\delta(n)+\delta(n-k+s)
\end{align*}
for $n < 3k-r$.

First we observe that $H(n)-H'(n) = 0$ when $n<3k-s-r$. Note that $3k-n-t \ge 3k-n-s > r$, hence $\tilde{b}_{3k-n-t} = \tilde{b}'_{3k-n-t}$ (note that $I \cap \{i>r\} = I' \cap \{i>r\}$). For $t \ge r$ one therefore has $b_t \tilde{b}_{3k-n-t}-b'_t \tilde{b}'_{3k-n-t} =0$. When $t<r$, observe that $3k-n-t \ge s+r-t > s$, hence $b_t \tilde{b}_{3k-n-t}-b'_t \tilde{b}'_{3k-n-t} =0$. It follows that $H(n)-H'(n)=-\delta(n)=0$ for $n<3k-s-r$.

If $b_t=1$ and $t \notin S$, one has $n-k+t < 3k-s-r$, hence $H(n-k+t)-H'(n-k+t) =0$. Let $p=\lfloor \frac{n-2k}{k-s} \rfloor$.  One then has
\begin{align*}
    &H(n)-H'(n) = -\delta(n) - \sum_{i=3k-n}^s b_i (H(n-k+i)-H'(n-k+i))\\
    &=-\delta(n) - \sum_{t=1}^p \big( H(n-t(k-s))-H'(n-t(k-s)) \big) \\
    &- \sum_{i \notin S } b_i \big( H(n-k+i)-H'(n-k+i) \big)\\
    &=-\delta(n) - \sum_{l=1}^p \Big( -\delta(n-l(k-s))+\delta \big( n-(l+1)(k-s) \big) \Big)\\
    &=-\delta(n) + \delta(n-k+s),
\end{align*}
where we observe that $n-(p+1)(k-s) \le 2k$.  It follows that for $n < 3k-r$ one has
\begin{align} \label{H less than k-r}
\begin{split}
H(n)-H'(n) = -\delta(n) +\delta(n-k+s).
\end{split}
\end{align}

For $2k<n<3k-r$ by using (\ref{H less than k-r}) one thus has
\begin{align} \label{Delta < 3k-r}
\begin{split}
    \Delta(n) &= -q^{n-2k}\sum_{t=1}^r q^t ( \tilde{b}_t - \tilde{b}'_t ) + \sum_{t=2k}^n q^{n-t} \delta(t) - \sum_{t=3k-s}^n q^{n-t} \delta(t-k+s),
\end{split}
\end{align}
where we have used the fact that
\begin{align*}
    \Delta(2k) = -\sum_{t=1}^r q^t ( \tilde{b}_t-\tilde{b}'_t )
\end{align*}
when $s=s'$.

We now fix an index $i$ with $1 \le i \le k$ such that $b_i=0$. For $n \ge k$ let us denote by $h^*(n)$ the solution to the recurrence relation $h^*(n) = \sum_{t=1}^k b^*_t (qh^*(n+t-1)-h^*(n+t))$ where 
\begin{align*}
    b^*_t = \begin{cases} b_t & t \ne i\\ 1 & t=i, \end{cases}
\end{align*}
subject to the initial conditions $h^*(n)=0$ for $n<k$ and $h^*(k)=1$.  Let $H^*(n) = qh^*(n-1)-h^*(n)$.

Denote
\begin{align*}
    \delta^*(n) = \sum_{t=1}^{3k-n-i} b^*_t \tilde{b}^*_{3k-n-t}-b_t \tilde{b}_{3k-n-t} = \tilde{b}^*_{3k-n-i}+b^*_{3k-n-i}.
\end{align*}
We will also denote
\begin{align*}
    \Delta^*(n) = h^*(n)-h(n).
\end{align*}
It is easy to see that 
\begin{align} \label{delta star vals}
\begin{split}
    &\delta^*(t) \begin{cases} =0, & t<3k-i-s\\ =2, & t=3k-i-s\\ =0, & t=3k-i-s+1, \end{cases}\\
    &0 \le \delta^*(t) \le 2 \textrm{ for } 3k-i-s+1 < t \le 3k-1.
\end{split}
\end{align}

For $n<3k-s$ equality (\ref{Delta < 3k-r}) becomes
\begin{align*}
    \Delta^*(n) &= -q^{n-2k+i} + \sum_{t=2k}^n q^{n-t} \left( \tilde{b}^*_{3k-t-i} + b^*_{3k-t-i} \right) \\
    & = -q^{n-2k+i} + \sum_{t=3k-s-i}^n q^{n-t} \left( \tilde{b}^*_{3k-t-i} + b^*_{3k-t-i} \right) < 0,
\end{align*}
where we have used the fact that $i \le k-s-2$.  If $3k-s \le n < 3k-i$ then
\begin{align} \label{delta < 3k-r}
\begin{split}
    \Delta^*(n) &= -q^{n-2k+i} + \sum_{t=2k}^n q^{n-t} \left( \tilde{b}^*_{3k-t-i} + b^*_{3k-t-i} \right) \\
    &- \sum_{t=3k-s}^n q^{n-t} \left( \tilde{b}^*_{3k-t-i+k-s} + b^*_{3k-t-i+k-s} \right) < 0.
\end{split}
\end{align}

Let $3k-i \le n \le 3k$.  One has 
\begin{align} \label{asdfg}
\begin{split}
    &\Delta^*(n) \le -q^{n-2k+i}  + \sum_{t=2k}^{3k-i-1} q^{n-t} \left( \tilde{b}^*_{3k-t-i} + b^*_{3k-t-i} \right) \\
    &- \sum_{t=3k-i}^n q^{n-t} \left(H^*(t)-H( t) \right) = -q^{n-2k+i}  + \sum_{t=2k}^{3k-i-1} q^{n-t} \left( \tilde{b}^*_{3k-t-i} + b^*_{3k-t-i} \right) \\
    &- \sum_{t=3k-i}^n q^{n-t} \Delta^*(t-k) + \sum_{t=3k-i}^n q^{n-t} H^*(t-k+i)\\
    &+ \sum_{t=3k-i}^n q^{n-t} \sum_{j=1}^s b_j \left( H^*( t-k+j)-H(t-k+j) \right).
\end{split}
\end{align}
Note that $t-k+j \le 2k+s < 3k-i$ hence $H^*(t-k+j)-H(t-k+j) = -\delta^*(t-k+j)+\delta^*(t-2k+j+s)$ when $t-k+j \ge 2k$.  

One has
\begin{align*}
    &\sum_{t=3k-j}^n q^{n-t} b^*_{4k-t-j-i} - \sum_{t=3k-j+(k-s)}^n q^{n-t} b^*_{4k-t-j-i+(k-s)} \\
    &=\sum_{t=3k-j}^n q^{n-t} b^*_{4k-t-j-i} - \sum_{t=3k-j}^{n-k+s} q^{n-t-k+s} b^*_{4k-t-j-i} \ge 0
\end{align*}
and similarly
\begin{align*}
    &\sum_{t=3k-j}^n q^{n-t} \tilde{b}^*_{4k-t-j-i} - \sum_{t=3k-j+(k-s)}^n q^{n-t} \tilde{b}^*_{4k-t-j-i+(k-s)} \ge 0.
\end{align*}
It follows that 
\begin{align} \label{LP -2}
\begin{split}
    &\sum_{t=3k-j}^n q^{n-t} (H^*(t-k+j)-H(t-k+j)) \\
    &= \sum_{t=3k-j}^n q^{n-t} \delta^*(t-k+j)-\sum_{t=3k-j+(k-s)}^n q^{n-t} \delta^*(t-k+j-(k-s)) \ge 0.
\end{split}
\end{align}
Using (\ref{LP -2}) one has
\begin{align} \label{LP -1}
\begin{split}
    &\sum_{t=3k-i}^n q^{n-t} \sum_{j=1}^s b_j \left( H^*(t-k+j)-H(t-k+j) \right)\\
    &=\sum_{j=1}^{i-1} b_j \sum_{t=3k-i}^{ 3k-j-1} q^{n-t} \left( \tilde{b}^*_{3k-t-j} - \tilde{b}_{3k-t-j} \right) - \sum_{j=1}^s \sum_{t=3k-j}^n q^{n-t} (H^*(t-k+j)-H(t-k+j))\\
    &\le \sum_{j=1}^{i-1} b_j \sum_{t=3k-i}^{ 3k-j-1} q^{n-t} \left( \tilde{b}^*_{3k-t-j} - \tilde{b}_{3k-t-j} \right)
\end{split}
\end{align}
for $n \le 3k$.

For $3k-i \le t < 3k$ note that 
\begin{align} \label{LP -2.5}
    \Delta^*(t-k) = - \sum_{l=3k-t}^i q^{t-3k+l} \left( b^*_l-b_l \right)
\end{align}
and recall
\begin{align*}
    \Delta^*(2k) = -\sum_{l=1}^i q^l (b^*_l - b_l),
\end{align*}
whence $\Delta^*(t-k) = -q^{t-3k+i}$ for $3k-i \le t \le 3k$.  Using $H^*(t-k+i) \le q^{t-3k+i}$ for $t-k+i \ge 2k$ and (\ref{LP -1}), equality (\ref{asdfg}) becomes
\begin{align} \label{LP 0}
\begin{split}
    &\Delta^*(n) \le -q^{n-2k+i}  + \sum_{t=2k}^{3k-i-1} q^{n-t} \left( \tilde{b}^*_{3k-t-i} + b^*_{3k-t-i} \right) \\
    &+\sum_{t=3k-i}^n q^{n-3k+i} + \sum_{t=3k-i}^n q^{n-3k+i} +\sum_{j=1}^{i-1} b_j \sum_{t=3k-i}^{ 3k-j-1} q^{n-t} \left( \tilde{b}^*_{3k-t-j} - \tilde{b}_{3k-t-j} \right)\\
    &\le -q^{n-2k+i} + 2 \sum_{t=3k-s-i}^{3k-i-1} q^{n-t} + 2(i+1) q^{n-3k+i} + q^{n-3k+2i}\\
    &\le -q^{n-2k+i} + q^{n-3k+s+i+2} + q^{n-3k+2i+1}+q^{n-3k+2i} \le 0
\end{split}
\end{align}
where we observe that $\tilde{b}^*_{3k-t-j}-\tilde{b}_{3k-t-j}=1$ if and only if $t=3k-i-j$ and where we have used the facts $i < r \le k-s-2$ and $2(i+1) \le q^{i+1}$.\\

Let $3k < n < 4k$.  Using (\ref{cancellation of non-prime periods}) one has
\begin{align} \label{LP 1}
\begin{split}
    &-\sum_{t=3k+1}^n q^{n-t} \left( H^*(t)-H(t) \right) =-\sum_{t=3k+1}^n q^{n-t} \Delta^*(t-k) + \sum_{t=3k+1}^n \sum_{j \in I} \Delta^*(t-2k+j)\\
    &-\sum_{t=3k+1}^n q^{n-t} \sum_{j \in I} \sum_{l=1}^{k-j} b_l \left( H^*(t-2k+j+l) - H(t-2k+j+l) \right)\\
    &+\sum_{t=3k+1}^n q^{n-t} H^*(t-k+i) - \sum_{t=3k+1}^n \sum_{j \in I} q^{n-t} H^*(t-2k+j+i).
\end{split}
\end{align}

For fixed $j$ and $u \le 4k-j-l-1$ one has
\begin{align} \label{LP -2.625}
\begin{split}
    &- \sum_{l=1}^{k-j} b_l \sum_{t=3k+1}^u q^{n-t} \left( H^*(t-2k+j+l)-H(t-2k+j+l) \right)\\
    &= - \sum_{l=1}^{k-j} b_l \sum_{t=3k+1}^u q^{n-t} \left( \tilde{b}^*_{4k-t-j-l}- \tilde{b}_{4k-t-j-l} \right) \le 0,
\end{split}
\end{align}
observing that $\tilde{b}^*_\iota \ge \tilde{b}_\iota$.  For $4k-j-l-1 \le u \le 5k-j-l-i-1$ one also has
\begin{align} \label{LP -2.75}
\begin{split}
    &-\sum_{l=1}^{k-j} \sum_{t=4k-j-l-1}^u q^{n-t} \left( H^*(t-2k+j+l)-H(t-2k+j+l) \right)\\
    &=\sum_{l=1}^{k-j} \sum_{t=4k-j-l-1}^u q^{n-t} \left( \delta^*(t-2k+j+l)-\delta^*(t-2k+j+l-k+s) \right)\\
    &\le \sum_{l=1}^{k-j} \sum_{t=5k-j-l-s-i}^u q^{n-t} \delta^*(t-2k+j+l) \le 2q^{n-4k+i+s+2}.
\end{split}
\end{align} 

For $3k-i \le n \le 3k$ one has
\begin{align*}
    &H^*(n)-H(n) = \Delta^*(n-k)-\sum_{t=1}^s b_t \left( H^*(n-k+t)- H(n-k+t) \right) - H^*(n-k+i)\\
    &\ge \Delta^*(n-k) - \sum_{t=1}^s \delta^*(n-2k+t+s) - H^*(n-k+i)\\
    &\ge -q^{n-3k+i}-2s-H^*(n-k+i).
\end{align*}
For $n \ge 4k-i$ one thus has
\begin{align} \label{LP -3}
\begin{split}
    &-\sum_{l=1}^{k-j} \sum_{t=5k-j-l-i}^n q^{n-t} \left( H^*(t-2k+j+l)-H(t-2k+j+l) \right)\\
    &\le -\sum_{l=1}^{k-j} \sum_{t=5k-j-l-i}^n q^{n-t} \left( -q^{t-5k+j+l+i}-2s-H^*(t-3k+j+l+i) \right)\\
    &\le q^{n-4k+2i} + q^{n-4k+s+i+2} + \sum_{l=1}^{k-j} \left( \sum_{t=5k-j-l-i}^{5k-j-l-i-1} q^{n-t} + \sum_{t=5k-j-l-i}^n q^{n-5k+j+l+i} \right)\\
    &\le q^{n-4k+2i} + q^{n-4k+s+i+2} + q^{n-4k+i+2} + q^{n-4k+2i}\\
    &\le 2q^{n-3k-4}+ q^{n-3k-1}+q^{n-3k-3} \le q^{n-3k-1}+q^{n-3k-2} + q^{n-3k-3},
\end{split}
\end{align}
where we have applied (\ref{LP -2.5}) and used the inequalities 
\begin{align} \label{LP 1.5}
    &i \le q^{i-1}, \hspace{7mm} i<r, \hspace{7mm} r \le k-s-2, \hspace{7mm} r \le \frac{k}{2}-1, \hspace{7mm} s \le k-3.
\end{align}

For $n< 4k-i$
\begin{align} \label{LP -3 E}
    &-\sum_{l=1}^{k-j} \sum_{t=5k-j-l-i}^n q^{n-t} \left( H^*(t-2k+j+l)-H(t-2k+j+l) \right) = 0
\end{align}
since the former sum is empty.  Using (\ref{LP -2.625}), (\ref{LP -2.75}), and (\ref{LP -3 E}), for $3k < n < 4k-i$ one has
\begin{align} \label{LP -3.25}
    &-\sum_{t=3k+1}^n q^{n-t} \sum_{j \in I} \sum_{l=1}^{k-j} b_l \left( H^*(t-2k+j+l) - H(t-2k+j+l) \right) \le \sum_{j \in I} 2 q^{n-4k+s+i+2}.
\end{align}
For any $3k < n < 4k$ note that
\begin{align} \label{LP -3.5}
\begin{split}
    &\sum_{j \in I} \sum_{t=3k+1}^n q^{n-t} H^*(t-2k+j+i)\\
    &=\sum_{j \in I} \left( \sum_{t=4k-j-i-s}^{4k-j-i-1} q^{n-t} \tilde{b}^*_{4k-t-j-i} + \sum_{t=4k-j-i}^n q^{n-t} H^*(t-2k+j+i) \right)\\
    &\le \sum_{j \in I} \left( q^{n-4k+j+i+s+1} + (n-3k) q^{n-4k+j+i} \right) \le q^{n-3k+s-2} + (n-3k)q^{n-3k-2}.
\end{split}
\end{align}

For $3k < n < 4k-i$ one has $\Delta^*(t-k) \ge -q^{t-3k+i}$.  Applying inequalities (\ref{LP -3.25}) and (\ref{LP -3.5}) to (\ref{LP 1}) we have
\begin{align*}
    &-\sum_{t=3k+1}^n q^{n-t} \left( H^*(t)-H(t) \right) \le -\sum_{t=3k+1}^n q^{n-t} \Delta^*(t-k) + (k-s) 2 q^{n-4k+s+i+2} \\
    &+ (n-3k) q^{n-3k+i} + (n-3k)q^{n-3k-2} + q^{n-3k+s-2}\\
    &\le (n-3k-1) q^{n-3k+i}  + (k-s) q^{n-4k+s+i+3} + (n-3k) q^{n-3k+i}\\
    &+ (n-3k)q^{n-3k-2}+q^{n-3k+s-2}\\
    &\le q^{n-3k+i+k/2} + q^{n-3k+i+2} + q^{n-3k+i+k/2} + q^{n-3k+k/2-2}+q^{n-3k+s-2}\\
    &\le 4q^{n-2k-3} + q^{n-2k-5},
\end{align*}
where we have used the inequality $n-3k \le k-1 \le q^{k/2}$, inequalities (\ref{LP 1.5}), and the fact that $k \ge 5$ when $s=s'$.

Similar to (\ref{LP 0}) one has
\begin{align*}
    &q^{n-2k}\Delta^*(2k) - \sum_{t=2k}^{3k-1} q^{n-t} (H^*(t)-H(t))\\
    &\le -q^{n-2k+i}+q^{n-3k+s+i+2}+q^{n-3k+2i+1}+q^{n-3k+2i}.
\end{align*}
It follows that for $3k < n < 4k-i$ we have
\begin{align*}
    &\Delta^*(n) \le  q^{n-2k}\left( -q^i + q^{-1}+q^{-5}+q^{-6}+4q^{-3}+q^{-5} \right) \le 0,
\end{align*}
where we have again used the inequalities (\ref{LP 1.5}).

For $3k-i \le n \le 3k$ it is easy to see that
\begin{align*}
    \Delta^*(n) &\ge -q^{n-2k+i} - \sum_{t=3k-i}^n q^{n-t} (H^*(t)-H(t))\\
    &\ge -q^{n-2k+i} -(n-3k+i)q^{n-2k}\\
    &\ge -q^{n-2k+i}-q^{n-2k+i}.
\end{align*}
For $4k-i \le n < 4k$ one thus has
\begin{align*}
    &-\sum_{t=3k+1}^n q^{n-t} \left( H^*(t)-H(t) \right) \le 4q^{n-2k-3} + q^{n-2k-5} + \sum_{t=4k-i}^n q^{n-3k+i}\\
    &+q^{n-3k-1}+q^{n-3k-2} + q^{n-3k-3}\\
    &\le 4q^{n-2k-3} + q^{n-2k-5} + q^{n-3k+2i-1}+q^{n-3k}\\
    &\le 4q^{n-2k-3} + 3q^{n-2k-5}
\end{align*}
where we have used (\ref{LP -3}).

For $4k-i \le n < 4k$ one thus has
\begin{align*}
    &\Delta^*(n) \le q^{n-2k}\left( -q^i + q^{-1}+q^{-5}+q^{-6}+4q^{-3}+3q^{-5} \right) \le 0.
\end{align*}
Given $w$ and $\iota \in I-\{s\}$, let $h_\iota(n)$ be the solution to the recurrence relation defined by 
\begin{align*}
    &h_\iota(n) = \sum_{t=1}^k c_t (qh_\iota(n-t-1)-h_\iota(n-t)), \hspace{5mm} c_t = \begin{cases} b_t & t>\iota\\ 0 & t=\iota\\ 1 & t<\iota \end{cases}.
\end{align*}
and $h^\iota(n)$ the solution to
\begin{align*}
    &h^\iota(n) = \sum_{t=1}^k c_t (qh^\iota(n-t-1)-h^\iota(n-t)), \hspace{5mm} c_t = \begin{cases} b_t & t>\iota\\ 1 & t=\iota\\ 0 & t<\iota \end{cases}.
\end{align*}
We have shown that $\Delta(b_i+1, n) \le \Delta(b_i,n)$ for $2k \le n < 4k$.  As we remarked, with minimal alterations to the calculations of this section one can show the inequality $\Delta(b_i-1,n) \ge \Delta(b_i,n)$ if $b_i=1$.  Thus $\Delta(n) \le \Delta^*(n)$ for $2k \le n < 4k$ where $h^*(n)$ is the result of setting $b_r=0$.

\section*{Appendix 2. The Upper Bound is Negative}

Let $h^r(n)$ and $h_r(n)$ be as defined in Appendix 1 above. We will show that $h^r(n)-h_r(n) \le 0$ for $2k \le n < 4k$.  Throughout this section we will let $h(n) = h^r(n)$ and $h'(n) = h_r(n)$ to avoid further burdening the notation.

For $t<3k-r$ one has
\begin{align*}
    &\delta(t) = \tilde{b}_{3k-t-r}-\sum_{l=1}^{r-1} \tilde{b}'_{3k-t-l} + \sum_{l=r+1}^{3k-t-1} b_l \left( \tilde{b}_{3k-t-l} - \tilde{b}'_{3k-t-l} \right)\\
    &=\tilde{b}_{3k-t-r}-\sum_{l=1}^{r-1} \tilde{b}'_{3k-t-l} + \sum_{l=r+1}^{3k-t-r-1} b_l \left( \tilde{b}_{3k-t-l} - \tilde{b}'_{3k-t-l} \right)+b_{3k-t-r} - \sum_{l=3k-t-r+1}^{3k-t-1} b_l\\
    &=\tilde{b}_{3k-t-r}-\sum_{l=1}^{r-1} \tilde{b}'_{3k-t-l} + b_{3k-t-r} - \sum_{l=3k-t-r+1}^{3k-t-1} b_l\\
    &=\tilde{b}_{3k-t-r}-\sum_{l=1}^{r-1} \tilde{b}'_{3k-t-l} + b_{3k-t-r} - \sum_{l=1}^{r-1} b_{3k-t-l}.
\end{align*}
It is thus easy to see that
\begin{align} \label{LP2 -1}
    &\delta(t) \begin{cases} =0, & t < 3k-s-r\\ =2, & t=3k-s-r\\ = -2, & t=3k-s-r+1 \end{cases} \delta(t) \begin{cases} \le 2, & t < 3k-r\\ \ge -2, & t<2k+s-r\\ \ge -2(r-1), & 2k+s-r \le t < 3k-r \end{cases}.
\end{align}

Using (\ref{LP2 -1}), for $n<3k-r$ one thus has
\begin{align*}
    -\sum_{t=3k-s}^n q^{n-t} \delta(t-k+s) \le 0
\end{align*}
and
\begin{align*}
    \sum_{t=2k}^n q^{n-t} \delta(t) \le 2 q^{n-3k+s+r}.
\end{align*}

Using (\ref{Delta < 3k-r}) and the equality $\Delta(2k)=-2$, for $2k \le n < 3k-r$ one thus has
\begin{align*}
    &\Delta(n) \le -q^{n-2k} \Delta(2k) + 2q^{n-3k+s+r} \le -2\left( q^{n-2k} - q^{n-3k+s+r} \right) < 0.
\end{align*}

Let $3k-r \le n \le 3k$.  One has
\begin{align} \label{LP2 1}
\begin{split}
    &\Delta(n) = q^{n-3k+r+1} \Delta(3k-r-1) - \sum_{t=3k-r}^n q^{n-t} (H(t)-H'(t))\\
    &= q^{n-3k+r+1} \Delta(3k-r-1) - \sum_{t=3k-r}^n q^{n-t} \Delta(t-k) \\
    &+ \sum_{t=3k-r}^n q^{n-t} \sum_{j=1}^s \left( b_j H(t-k+j)-b'_j H'(t-k+j) \right).
\end{split}
\end{align}
Note that
\begin{align*}
    &\sum_{t=3k-r}^n q^{n-t} \sum_{j=1}^s \left( b_j H(t-k+j)-b'_j H'(t-k+j) \right)\\
    &=\sum_{t=3k-r}^n \left( q^{n-t} H(t-k+r) -\sum_{j=1}^{r-1} q^{n-t} H'(t-k+j) \right)\\
    &+ \sum_{j=r+1}^s b_j \sum_{t=3k-r}^n q^{n-t} \left(H(t-k+j)-H'(t-k+j) \right)\\
    &\le \sum_{t=3k-r}^n q^{n-t} H(t-k+r)+ \sum_{j=r+1}^s b_j \sum_{t=3k-r}^n q^{n-t} \left(H(t-k+j)-H'(t-k+j) \right).
\end{align*}
For $r < j \le s$ and $3k-r \le t \le n$ one has $2k < t-k+j < 3k-r$, and again using (\ref{LP2 -1}) we have
\begin{align*}
    &\sum_{j=r+1}^s b_j \sum_{t=3k-r}^n q^{n-t} \left(H(t-k+j)-H'(t-k+j) \right)\\
    &= -\sum_{j=r+1}^s b_j \sum_{t=3k-r}^n q^{n-t} \left( \delta(t-k+j)-\delta(t-2k+j+s) \right)\\
    &\le 2(r-1) \sum_{j=r+1}^s b_j q^{n-3k+r+1} + 2 \sum_{j=r+1}^s b_j q^{n-5k+2s+r+j}\\
    &\le 2(r-1) q^{n-3k+s} + q^{n-5k+3s+r+2} \le q^{n-3k+s+r-1} + q^{n-4k+2s}\\
    &\le q^{n-3k+s+r-1} + q^{n-3k+s-3}, 
\end{align*}
where we have applied the inequalities $r+s \le k-2$ and $s \le k-3$.  Inequality (\ref{LP2 1}) thus becomes
\begin{align} \label{up to 3k}
\begin{split}
    &\Delta(n) = q^{n-3k+r+1} \Delta(3k-r-1) - \sum_{t=3k-r}^n q^{n-t} (H(t)-H'(t))\\
    &\le q^{n-3k+r+1} \Delta(3k-r-1) - \sum_{t=3k-r}^n q^{n-t} \Delta(t-k)+ q^{n-3k}((r+1)q^r + q^{s+r-1}+q^{s-3}\\
    &\le -2q^{n-2k}+2q^{n-3k+r+s}+\sum_{t=3k-r}^n q^{n-t} + q^{n-3k}(q^{2r}+q^{s+r-1}+q^{s-3})\\
    &\le -2q^{n-2k}+q^{n-3k}(q^{s+r+1}+q^{r+1}+q^{2r}+q^{s+r-1}+q^{s-3})\\ 
    &\le -q^{n-2k},
\end{split}
\end{align}
where we have used the inequality $\Delta(3k-r-1) \le -2\left( q^{k-r-1}-q^{s-1} \right)$ and the fact that $\Delta(n)=-1$ for $2k-r \le n \le 2k$.\\

Let $3k < n < 4k$.  Denote $I_t = I \cap \{t, \dots, s\}$.  One has
\begin{align*}
    &-\sum_{t=3k+1}^n q^{n-t} (H(t)-H'(t)) = -\sum_{t=3k+1}^n q^{n-t} \Delta(t-k) + \sum_{t=3k+1}^n q^{n-t} \sum_{j \in I_{r+1}} \Delta(t-2k+j) \\
    &- \sum_{t=3k+1}^n q^{n-t} \sum_{j \in I_{r+1}} \sum_{l=1}^{k-j} \left( b_l H(t-2k+j+l)-b'_l H'(t-2k+j+l) \right)\\
    &+\sum_{t=3k+1}^n q^{n-t} H(t-k+r)-\sum_{j=1}^{r-1} \sum_{t=3k+1}^n q^{n-t} H'(t-k+j),
\end{align*}
hence 
\begin{align} \label{last case}
\begin{split}
    &-\sum_{t=3k+1}^n q^{n-t} (H(t)-H'(t)) \le -\sum_{t=3k+1}^n q^{n-t} \Delta(t-k)+\sum_{t=3k+1}^n q^{n-t} H(t-k+r)\\
    &- \sum_{t=3k+1}^n q^{n-t} \sum_{j \in I_{r+1}} \sum_{l=1}^{k-j} \left( b_l H(t-2k+j+l)-b'_l H'(t-2k+j+l) \right).
\end{split}
\end{align}

Let $n<4k-r$.  Note that $t-2k+j+l < 3k-r$.  For fixed $j$, if $r<l \le k-j$ one has
\begin{align*}
    &-\sum_{t=3k+1}^n q^{n-t}  \left( b_l H(t-2k+j+l)-b'_l H'(t-2k+j+l) \right)\\
    &= -b_l \sum_{t=3k+1}^n q^{n-t} \left( H(t-2k+j+l)- H'(t-2k+j+l) \right)\\
    &\le -b_l \sum_{t=3k+1}^{4k-j-l-1} q^{n-t} \left( \tilde{b}_{4k-t-j-l}  - \tilde{b}'_{4k-t-j-l} \right) + b_l \sum_{t=4k-j-l}^n q^{n-t} \delta(t-2k+j+l)\\
    &-b_l \sum_{t=4k-j-l}^n q^{n-t} \delta(t-3k+j+l+s).
\end{align*}
For any $u \le 4k-j-l-1$ observe that
\begin{align*}
    &-\sum_{j \in I_{r+1}} \sum_{l=r+1}^{k-j} b_l \sum_{t=3k+1}^{u} q^{n-t} \left( \tilde{b}_{4k-t-j-l}  - \tilde{b}'_{4k-t-j-l} \right)\\
    &=-\sum_{j \in I_{r+1}} \sum_{l=r+1}^{k-j} b_l \sum_{t=4k-j-l-r}^{u} q^{n-t}  \left( \tilde{b}_{4k-t-j-l}  - \tilde{b}'_{4k-t-j-l} \right)\\
    &=-\sum_{j \in I_{r+1}} \sum_{l=r+1}^{k-j} b_l \sum_{t=4k-j-l-u}^r q^{n-4k+j+l+t} \left( \tilde{b}_t-\tilde{b}'_t \right)  \le 0.
\end{align*}
Similarly, for any $u \le 5k-j-l-r-1$ one has
\begin{align*}
    &-\sum_{j \in I_{r+1}} \sum_{l=1}^{k-j} b_l \sum_{t=4k-j-l}^u q^{n-t} \delta(t-3k+j+l+s) \le 0.
\end{align*}
It follows that
\begin{align} \label{stuff 1}
\begin{split}
    &-\sum_{t=3k+1}^n q^{n-t} \sum_{j \in I_{r+1}} \sum_{l=r+1}^{k-j} \left( b_l H(t-2k+j+l)-b'_l H'(t-2k+j+l) \right)\\
    &\le \sum_{j \in I_{r+1}} \sum_{l=r+1}^{k-j} b_l \sum_{t=4k-j-l}^n q^{n-t} \delta(t-2k+j+l).
\end{split}
\end{align}

Noting that $r < k-s \le k-j$ for any $j \in I$, one has
\begin{align*}
    &-\sum_{t=3k+1}^n q^{n-t} \sum_{j \in I_{r+1}} \sum_{l=1}^r \left( b_l H(t-2k+j+l) - b'_l H'(t-2k+j+l) \right)\\
    &=-\sum_{j \in I_{r+1}} \left( \sum_{t=3k+1}^{4k-j-r-1} q^{n-t} \tilde{b}_{4k-t-j-r} - \sum_{l=1}^{r-1} \sum_{t=3k+1}^{4k-j-l-1} q^{n-t} \tilde{b}'_{4k-t-j-l} \right)\\
    &-\sum_{j \in I_{r+1}} \left( \sum_{t=4k-j-r}^n q^{n-t} H(t-2k+j+r) - \sum_{l=1}^{r-1} \sum_{t=4k-j-l}^n q^{n-t} H'(t-2k+j+l) \right)\\
    &\le \sum_{j \in I_{r+1}} \sum_{l=1}^{r-1} \sum_{t=3k+1}^{4k-j-l-1} q^{n-t} \tilde{b}'_{4k-t-j-l} \\
    &-\sum_{j \in I_{r+1}} \sum_{t=4k-j-r}^n q^{n-t} \left( H(t-2k+j+r)-H'(t-2k+j+r) \right).
\end{align*}
It is easy to see that
\begin{align*}
    &-\sum_{j \in I_{r+1}} \sum_{t=4k-j-r}^n q^{n-t} \left( H(t-2k+j+r)-H'(t-2k+j+r) \right)\\
    &\le \sum_{j \in I_{r+1}} \sum_{t=4k-j-r}^n q^{n-t} \delta(t-2k+j+r).
\end{align*}
We thus have
\begin{align} \label{stuff 2}
\begin{split}
    &-\sum_{t=3k+1}^n q^{n-t} \sum_{j \in I_{r+1}} \sum_{l=1}^r \left( b_l H(t-2k+j+l) - b'_l H'(t-2k+j+l) \right)\\
    &\le \sum_{j \in I_{r+1}} \sum_{l=1}^{r-1} \sum_{t=3k+1}^{4k-j-l-1} q^{n-t} \tilde{b}'_{4k-t-j-l} + \sum_{j \in I_{r+1}} \sum_{t=4k-j-r}^n q^{n-t} \delta(t-2k+j+r).
\end{split}
\end{align}
Using (\ref{stuff 1}) and (\ref{stuff 2}) we have
\begin{align} \label{stuff 3}
\begin{split}
    &-\sum_{t=3k+1}^n \sum_{j \in I_{r+1}} \sum_{l=1}^{k-j} q^{n-t} \left( b_l H(t-2k+j+l)-b'_l H'(t-2k+j+l) \right)\\
    &\le \sum_{j \in I_{r+1}} \sum_{l=r}^{k-j} b_l \sum_{t=5k-j-l-s-r}^n q^{n-t} \delta(t-2k+j+l) + \sum_j \sum_{l=1}^{r-1} \sum_{t=4k-j-l-s}^{4k-j-l-1} q^{n-t} \tilde{b}'_{4k-t-j-l}\\
    &\le 2 \sum_j \sum_{l=1}^{k-j} q^{n-5k+s+r+j+l} + \sum_j \sum_{l=1}^{r-1} q^{n-4k+s+j+l+1}\\
    &\le 2 \sum_j  q^{n-4k+s+r+1} + \sum_j q^{n-4k+s+j+r+1} \le q^{k-s-r-1} q^{n-4k+s+r+1}+q^{n-4k+2s+r+2}\\
    &\le q^{n-3k}+q^{n-3k+s},
\end{split}
\end{align}
using the fact that $2|I_{r+1}| \le 2(k-s-r-1) \le q^{k-s-r-1}$.  

For $2k < n < 3k-r$, using equality (\ref{Delta < 3k-r}) we easily obtain the lower bound
\begin{align*}
    \Delta(n) &= -q^{n-2k}\Delta(2k) + \sum_{t=2k}^n q^{n-t} \delta(t) - \sum_{t=2k}^{n-k+s} q^{n-t-k+s} \delta(t)\\
    &\ge -2q^{n-2k} + \sum_{t=n-k+s+1}^n q^{n-t} \delta(t).
\end{align*} 
If $3k-s-r > n-k+s$ then $\sum_{t=n-k+s+1}^n q^{n-t} \delta(t) \ge 0$.  If $3k-s-r \le n-k+s$ then $\sum_{t=2k}^n q^{n-t} \delta(t) \ge 0$ and $-\sum_{t=2k}^{n-k+s} q^{n-t-k+s} \delta(t) \ge -2q^{n-4k+2s+r}$.  We thus obtain the lower bound 
\begin{align} \label{Delta < 2k}
    \Delta(n) \ge -2q^{n-2k}-2q^{n-4k+2s+r}.
\end{align}
Applying (\ref{stuff 3}) and (\ref{Delta < 2k}) to (\ref{last case}) one has
\begin{align*}
    &-\sum_{t=3k+1}^n q^{n-t} (H(t)-H'(t)) \le -\sum_{t=3k+1}^n q^{n-t} \Delta(t-k)+\sum_{t=3k+1}^n q^{n-3k+r}+q^{n-3k}+q^{n-3k+s}\\
    &\le 2\sum_{t=3k+1}^n q^{n-3k} +2\sum_{t=3k+1}^nq^{n-5k+2s+r}+ (k-r-1) q^{n-3k+r} +q^{n-3k}+q^{n-3k+s}\\
    &\le 2(k-r-1)q^{n-3k}+2(k-r-1)q^{n-5k+2s+r} + q^{n-2k-2} + q^{n-3k}+q^{n-3k+s}\\
    &\le q^{n-2k-r-1}+q^{n-4k+2s-1} + q^{n-2k-2} + q^{n-3k}+q^{n-3k+s}\\
    &\le 2q^{n-2k-2} + q^{n-2k-7}+q^{n-3k}+q^{n-2k-3}.
\end{align*}
One thus has
\begin{align*}
    \Delta(n) \le q^{n-2k}(-1+q^{-1}+q^{-7}+q^{-k}+q^{-3}) \le 0.
\end{align*}

Let $4k-r \le n < 4k$.  By calculations similar to those above, one has
\begin{align*}
    &-\sum_{j \in I_{r+1}} \sum_{l=1}^{k-j} \sum_{t=3k+1}^n q^{n-t}  \left( b_l H(t-2k+j+l)-b'_l H'(t-2k+j+l) \right)\\
    &\le \sum_{j \in I_{r+1}} \sum_{l=1}^{r-1} \sum_{t=3k+1}^{4k-j-l-1} q^{n-t} \tilde{b}'_{4k-t-j-l} + \sum_{j \in I_{r+1}} \sum_{l=r}^{k-j} b_l \sum_{t=4k-j-l}^{5k-r-j-l-1} q^{n-t} \delta(t-2k+j+l)\\
    &-\sum_{j \in I_{r+1}} \sum_{l=r}^{k-j} b_l \sum_{t=5k-r-j-l}^n q^{n-t} \left( H(t-2k+j+l)-H'(t-2k+j+l) \right).
\end{align*}
For $3k-r \le n < 3k$ observe that
\begin{align*}
    &H(n)-H'(n) \ge \Delta(n-k)-H(n-k+r)+\sum_{t=r+1}^s ( \delta(n-k+t)-\delta(n-2k+t+s) )\\
    &\ge -2 - q^{n-3k+r} + \sum_{t=r+1}^s (\delta(n-k+t)- \delta(n-2k+t+s)).
\end{align*}
It follows that
\begin{align} \label{INE 1}
\begin{split}
    &-\sum_{j \in I_{r+1}} \sum_{l=r}^{k-j} b_l \sum_{t=5k-r-j-l}^n q^{n-t} \left( H(t-2k+j+l)-H'(t-2k+j+l \right)\\
    &\le \sum_{j \in I_{r+1}} \sum_{l=r}^{k-j} b_l \left( q^{n-5k+r+j+l+2}+(n-5k+r+j+l+1)q^{n-5k+j+l+r} \right.\\
    &\left. +2(r-1)q^{n-5k+r+j+l+2}+2q^{n-7k+3s+j+l+r+1} \right)\\
    &\le q^{n-3k-s+1}+q^{n-3k+r-s-2}+q^{n-3k+r-s-1}+q^{n-5k+2s+1} \le q^{n-3k}.
\end{split}
\end{align}
One has
\begin{align} \label{INE 2}
\begin{split}
    &\sum_{j \in I_{r+1}} \sum_{l=r}^{k-j} b_l \sum_{t=4k-j-l}^{5k-r-j-l-1} q^{n-t} \delta(t-2k+j+l) \le 2\sum_{j \in I_{r+1}} \sum_{l=1}^{k-j} b_l q^{n-5k+j+l+s+r}\\
    & \le 2(k-s-r-1) q^{n-4k+s+r+1} \le q^{n-3k}
\end{split}
\end{align}
and
\begin{align} \label{INE 3}
\begin{split}
    &\sum_{j \in I_{r+1}} \sum_{l=1}^{r-1} \sum_{t=3k+1}^{4k-j-l-1} q^{n-t} \tilde{b}'_{4k-t-j-l} \le \sum_{j \in I_{r+1}} \sum_{l=1}^{r-1} \sum_{t=4k-s-j-l}^{4k-j-l-1} q^{n-t} b'_{4k-t-j-l}\\
    &\le \sum_j \sum_{l=1}^{r-1} q^{n-4k+s+j+l+1} \le \sum_{j \in I_{r+1}} q^{n-4k+s+j+r+1} \le q^{n-4k+2s+r+2} \le q^{n-3k+s}.
\end{split}
\end{align}
Finally, for $3k-r \le n < 3k$, using (\ref{LP2 1}) and (\ref{Delta < 2k}) one has the lower bound
\begin{align} \label{INE 4}
\begin{split}
    &\Delta(n) \ge -2 (q^{n-2k}+q^{n-4k+2s+r})\\
    &+ \sum_{t=3k-r}^n q^{n-t}\left( H(t-k+r) - \sum_{l=1}^{r-1} H'(t-k+l) - \sum_{l=r+1}^s b_l \delta(t-k+l) \right)\\
    &\ge -2 (q^{n-2k}+q^{n-4k+2s+r}) - \sum_{l=1}^{r-1} \sum_{t=3k-r}^{3k-l-1} q^{n-t} \tilde{b}'_{3k-t-l} - \sum_{l=r}^s \sum_{t=3k-r}^n b_l q^{n-t} \delta(t-k+l)\\
    &\ge -2 (q^{n-2k}+q^{n-4k+2s+r}) - \sum_{l=1}^{r-1} q^{n-3k+r+1} - 2\sum_{l=r}^s q^{n-4k+s+r+l}\\
    &\ge -2 (q^{n-2k}+q^{n-4k+2s+r}) - q^{n-3k+2r-1} - q^{n-4k+2s+r+2}.
\end{split}
\end{align}

Applying (\ref{INE 1}) through (\ref{INE 4}) to (\ref{last case}), for $4k-r \le n < 4k$ one thus has
\begin{align*}
    &\Delta(n) \le -q^{n-2k} + 2 \sum_{t=3k+1}^n (q^{n-3k} + q^{n-5k+2s+r}) + \sum_{t=4k-r}^n (q^{n-4k+2r-1} + q^{n-5k+2s+r+2})\\
    &+\sum_{t=3k+1}^n q^{n-t} H(t-k+r)+2q^{n-3k}+q^{n-3k+s}\\
    &\le -q^{n-2k} + 2(k-1)(q^{n-3k}+q^{n-4k+s-2})+(r-1)(q^{n-3k-5}+q^{n-4k+s})\\
    &+(k-1)q^{n-3k+r} + 2q^{n-3k}+q^{n-3k+s}\\
    &\le q^{n-2k}(-1+q^{-2}+q^{-7}+q^{-11}+q^{-k-4}+q^{-1}+2q^{-k}+q^{-3}) < 0
\end{align*}
where we have used (\ref{LP 1.5}) and the inequalities $k \ge 5$, $(k-1) \le q^{k-3}$, and $k-1 \le q^{k/2}$.

It follows that $\tilde{\Delta}(n) < 0$ for $n < 4k$.  In combination with our results from Section 6, one has
\begin{align*} 
    &\Delta(n)
    < 0 \textrm{ for } n<4k \textrm{ when } s=s' \textrm{ and } k=k'.
\end{align*}

\section*{Acknowledgments}

This work was partially supported by the NSF grant DMS-1600568.

\end{document}